\documentclass[12pt,a4paper,reqno]{amsart}
\usepackage{srcltx}

\usepackage{amsmath, amssymb, cases, color,bbm}
\usepackage{hyperref}
\usepackage[capitalise]{cleveref}
\usepackage{vmargin}
\setmarginsrb{2cm}{1cm}{2cm}{3cm}{1cm}{1cm}{2cm}{2cm}

\usepackage{amsmath, amssymb,amscd, }
\usepackage{amsfonts}
\usepackage{mathrsfs}
\usepackage{graphicx}

 \usepackage{upref}
\hypersetup{linkcolor=blue, colorlinks=true,citecolor = red}
\newtheorem{theorem}{Theorem}[section]
\newtheorem{lemma}[theorem]{Lemma}
\newtheorem{proposition}[theorem]{Proposition}

\theoremstyle{definition}
\newtheorem{definition}[theorem]{Definition}

\theoremstyle{remark}
\newtheorem{remark}[theorem]{Remark}

\numberwithin{equation}{section}

\newcommand{\ost}{\Omega_{S}(t)}
\newcommand{\oso}{\Omega_{S}(0)}
\newcommand{\oft}{\Omega_{F}(t)}
\newcommand{\ofo}{\Omega_{F}(0)}
\newcommand{\rt}{\mathbb{R}^{3}}
\newcommand{\ct}{\mathbb{C}^{3}}

\newcommand{\ds}{\displaystyle}
\newcommand{\poso}{\partial\oso}
\newcommand{\post}{\partial \ost}
\newcommand{\mx}{\mathcal{X}}
\newcommand{\my}{\mathcal{Y}}
\newcommand{\mz}{\mathcal{Z}}
\newcommand{\md}{\mathcal{D}}

\newcommand{\mr}{\mathcal{R}}
\newcommand{\mpp}{\mathcal{P}}
\newcommand{\mv}{\mathcal{V}}
\newcommand{\dg}{{\rm d} \gamma}

\begin{document}

\title[]{ $L^p$-$L^q$ Maximal Regularity for some Operators  Associated with Linearized  Incompressible Fluid-Rigid Body Problems}

 \date{\today}

\author{D Maity}
\address{Institut de Math\'ematiques, Universit\'e de Bordeaux, Bordeaux INP, CNRS
F-33400 Talence, France}
\email{debayan.maity@u-bordeaux.fr}
\author{M. Tucsnak}
\address{Institut de Math\'ematiques, Universit\'e de Bordeaux, Bordeaux INP, CNRS
F-33400 Talence, France}
\email{marius.tucsnak@u-bordeaux.fr}

\begin{abstract}
We study an unbounded operator arising naturally after linearizing the system modelling the motion of a rigid body in a viscous incompressible fluid. We show that this operator is $\mr$ sectorial in $L^q$ for every $q\in (1,\infty)$, thus it has the maximal $L^p$-$L^q$ regularity property. Moreover, we show that the generated semigroup is exponentially stable with respect to the $L^q$ norm. Finally, we use the results to prove the global existence for small initial data, in an $L^p$-$L^q$ setting, for the original nonlinear problem.
\end{abstract}

\maketitle


%\tableofcontents

%%%%%%%%%%%%%%%%%%%%%%%%%%%%%%%%%%%%%%%%%%%%%%%%%%%%%%%%%%%%%%%%%%%%%%

\section{Introduction and main results}\label{sec_intro}

The aim of this work is to show that the semigroup associated to the equations obtained  by linearizing some systems modelling fluid-structure interactions has the maximal $L^p$-$L^q$ regularity property. This result can be seen as an improvement of those in \cite{TT1,TT2}, where the result was proved in a Hilbert space setting and in \cite{WX11}, where it has been shown that the corresponding semigroup is analytic in $L^{\frac 65}(\mathbb{R}^3)\cap L^q(\mathbb{R}^3)$, for $q\geqslant 2$. We then apply this result in proving a global existence and uniqueness result, for small data, for the
original nonlinear problem. Such a result seems new in an  $L^p$-$L^q$ setting. The references \cite{TT1,TT2} and \cite{Gei13} contain closely related results  and methods which are often used in  the present paper.

Let us first remind the original free boundary system which motivates this work. We will come back to this system later on, in order to prove the global existence of solutions for small initial data. The smallness is measured in a Besov space and the solutions lie in function spaces which are $L^{p}$ with respect to time and $L^{q}$ with respect to the space variable. As far as we know, global existence results of this type were known only in an $L^2$ setting.

Consider a rigid body immersed in a viscous incompressible fluid and moving under the action of forces exerted by the fluid only.  At time $t \geqslant 0,$ this solid occupies a smooth bounded domain $\Omega_{S}(t).$ The fluid and rigid body are contained in a bounded domain $\Omega \subset \rt$ with smooth boundary $\partial \Omega.$ We assume that there exists a constant $\alpha$ with
\begin{align} \label{eq:ini-d}
\mathrm{dist} \left(\oso,\partial\Omega\right) \geqslant \alpha > 0.
\end{align}
 At any time  $t \geqslant 0,$ we denote by  $\Omega_{F}(t) = \Omega \setminus \overline \ost$ the  domain occupied by the fluid. We assume that the motion of the fluid is governed by the incompressible Navier-Stokes equations, whereas the motion of the structure is governed by the balance equation for linear and angular momentum.  The full system of equations modelling the rigid body inside  the fluid can be written as
 \begin{alignat}{2} \label{eq:mainsys}
& \partial_{t} u + ( u \cdot \nabla) u - \mathrm{div} \ \sigma(u,\pi) = 0, &  \quad t \in (0,\infty), \  x\in \oft, \notag \\
& \mathrm{div} \ u =0,  &  \quad t \in (0,\infty), \  x\in \oft, \notag  \\
&  u = 0, & \quad t \in (0,\infty), \  x\in \partial\Omega, \notag  \\
& u = a'(t) + \omega(t) \times (x - a(t)), & t \in (0,\infty), \  x\in \partial\ost, \\
& m a''(t) = -\int_{\post} \sigma(u,\pi) n \ {\rm d} \gamma, & t \in (0,\infty), \notag\\
& J \omega'(t)  = (J\omega) \times \omega - \int_{\post} (x- a(t)) \times \sigma(u,\pi)n  \ \dg, & t \in (0,\infty), \notag \\
& u(0,x) = u_{0}(x) & x \in \Omega_{F}(0), \notag \\
& a(0) = 0, \quad a'(0) = \ell_{0}, \quad \omega(0) = \omega_{0}.  \notag
 \end{alignat}
 In the above equations, $u(t,x)$ denote the velocity of the fluid, $\pi(t,x)$ denote the pressure of the fluid, $a(t)$ denote the position of the centre of the mass and $\omega(t)$ denote the angular velocity of the rigid body.  The domain $\Omega_{S}(t)$ is defined by
\[ \ost = a(t) + Q(t)y, \quad \forall y \in \Omega_{S}(0), \; \forall t > 0, \]
where $Q(t) \in \mathbb{M}_{3\times 3}(\mathbb{R})$ is the  orthogonal matrix giving the orientation of the solid. More precisely, $\omega(t)$ and $Q(t)$ are related to each other through the following relation
\begin{align} \label{eqQ}
\dot Q(t)Q(t)^{-1} y = A(\omega(t)) y =  \omega(t) \times y, \quad \forall y \in \mathbb{R}^{3}, \quad Q(0) = I_{3},
\end{align}
where the skew-symmetric matrix $A(\omega)$ is given by
\begin{align*}
A(\omega) = \begin{pmatrix}
0 & -\omega_{3} & \omega_{2} \\ \omega_{3} & 0 & -\omega_{1} \\ -\omega_{2} & \omega_{1} & 0
\end{pmatrix}, \qquad  \omega \in \mathbb{R}^{3}.
\end{align*}

The constant $m > 0$ denote  the mass of the rigid structure and $J(t) \in \mathcal{M}_{3 \times 3} (\mathbb{R})$ its tensor of inertia at time $t.$ This tensor is given by
\begin{align} \label{tensor}
J(t) a \cdot b = \int_{\oso} \rho_{S}(y) (a \times Q(t)y) \cdot (b \times Q(t)y) \ dy, \qquad \forall a,b \in \rt,
\end{align}
where $\rho_{S} > 0$ is the density of the structure.  One can check that
\begin{align}
J(t) a \cdot a\geqslant C_{J} |a|^{2} > 0,
\end{align}
where $C_{J}$ is independent of $t > 0.$ In the above,  we have denoted by $\post$ the boundary of the rigid structure at time $t$ and by $n(t,x)$ the unit normal to $\partial \ost$ at the point $x$
directed towards the interior of the rigid body. The Cauchy stress tensor  $\sigma(u,\pi)$ is given by
\begin{align}
\sigma(u,\pi) = -\pi I_{3} + 2\nu \varepsilon(u), \quad \varepsilon(u) = \frac{1}{2}\left(\nabla u + \nabla u ^{\top} \right),
\end{align}
where the positive constant $\nu $ is the viscosity of the fluid.

Linearizing the above equations around the zero solution we obtain a system coupling Stokes equations in a fixed domain and an  ODE system. The corresponding equations read as
\begin{alignat}{2} \label{fsi-l}
&\partial_{t} u - \nu \Delta u + \nabla \pi  = f, \quad \mathrm{div} \ u = 0, & \quad t \in (0,\infty), \  y\in \ofo, \notag \\
& u = 0,  & t \in (0,\infty), \  y\in \partial\Omega,\notag  \\
&  u  = \ell  + \omega \times y, & \quad t \in (0,\infty), \  y\in \partial\oso, \notag \\
&  m \ell' =  -  \int_{\partial \Omega_{S}(0)}  \sigma(u,\pi) n \ d\gamma  +  g_{1}, & \quad t \in (0,\infty), \\
&  J(0) \omega' =  - \int_{\partial \Omega_{S}(0)} y \times   \sigma(u,\pi)n \ d\gamma + g_{2}, &  \quad t \in (0,\infty), \notag \\
&  u (0,y) = u_{0}(y), &  y \in  \Omega_{F}(0),  \notag \\
&  \ell(0) = \ell_{0}, \quad  \omega(0) = \omega_{0},  \notag
\end{alignat}
where $n$ is the unit normal to $\poso$ directed towards the interior of the rigid body.

Let us now define  the operator associated with the above linear fluid-structure interaction problem, which was first introduced in \cite{TT1,TT2}. The idea is to extend the fluid velocity $u$ by $\ell(t) + \omega(t) \times y$ in $\Omega_{S}(0).$ More precisely,  for any $1 <  q < \infty$  we define
\begin{align}
& \mathbb{H}^{q}(\Omega) = \left\{ \varphi \in L^{q}(\Omega)^{3} \mid \mathrm{div} \ \varphi = 0 \mbox{ in } \Omega, \quad \varepsilon(\varphi) = 0 \mbox{ in } \oso, \quad \varphi \cdot n = 0 \mbox{ on } \partial\Omega \right\}
\end{align}

We define
\begin{align}
\mathcal{D}(\mathbb{A}) = \left\{ \varphi \in W^{1,q}_{0}(\Omega)^{3} \mid \varphi|_{\ofo} \in W^{2,q}(\ofo)^{3}, \quad \mathrm{div} \ \varphi = 0 \mbox{ in } \Omega, \quad \varepsilon(\varphi) = 0 \mbox{ in } \oso\right\}.
\end{align}
For all $v \in \mathcal{D}(\mathbb{A})$ we set
\begin{align*}
\mathcal{A}v = \begin{cases}
- \nu \Delta v & \mbox{ in } \ofo, \\
\ds  2\nu m^{-1} \int_{\poso} \varepsilon(v) n \ d \gamma  + \left(2\nu J(0)^{-1} \int_{\poso} y \times \varepsilon(v)n \ d\gamma \right) \times y & \mbox{ in } \oso,
\end{cases}
\end{align*}
and
\begin{align} \label{sg-TT}
\mathbb{A}v = \mathbb{P} \mathcal{A} v,
\end{align}
where $\mathbb{P}$ is the  projection from $L^{q}(\Omega)^{3}$ onto $\mathbb{H}^{q}(\Omega).$ The existence of such projector $\mathbb{P}$ can be found in \cite[Theorem 2.2]{WX11}.

Takahashi and Tucsnak \cite{TT1}   proved that  the operator $\mathbb{A}$ defined above generates an analytic semigroup on $\mathbb{H}^{2}(\Omega)$ when $\Omega = \mathbb{R}^{2}.$
When  $\Omega$ is a smooth bounded domain in  $\mathbb{R}^{2}$ the same result was proved by Takahashi \cite{TT2}. Later, Wang and Xin \cite{WX11} proved that the operator $\mathbb{A}$ generates an analytic semigroup on
$\mathbb{H}^{6/5}(\rt) \cap \mathbb{H}^{q}(\rt)$ if $q \geqslant 2$ and when the solid is a ball in $\rt$ the operator $\mathbb{A}$ generates an analytic semigroup on $\mathbb{H}^{2}(\rt) \cap \mathbb{H}^{q}(\rt)$ if $q \geqslant 6.$ In this article, as a corollary of our main result, we prove that the operator $\mathbb{A}$ generates an analytic semigroup on $\mathbb{H}^{q}(\Omega)$ for any $1 < q < \infty.$

Before we state our main result, we introduce the notion of maximal $L^{p}$-regularity. Let us consider the following Cauchy  problem:
\begin{equation} \label{eq:ab-cauchy}
z'(t) = A z(t) + f(t), \quad z(0) = z_{0},
\end{equation}
where $A$ is a closed, linear  densely defined unbounded operator in a Banach space $\mx$ with domain $\md(A),$  $f : \mathbb{R}^{+} \mapsto \mx$ is a locally integrable function and $z_{0} \in \mx.$

\begin{definition}
We say $A$ has maximal $L^{p}$-regularity property for $1 < p < \infty,$ on $[0,T),$ $0 < T \leqslant \infty,$ if for  $z_{0} = 0$ and for every $f \in L^{p}(0,T;\mx)$  there exists a unique $z\in W^{1,p}_{\rm loc}[0,\infty);\mx)$ satisfying \eqref{eq:ab-cauchy} almost everywhere and such that $z'$ and $Az$ belong to $L^{p}(0,T;\mx).$ We denote the class of all such operators by $\mathcal{MR}_{p}([0,T);\mx).$
\end{definition}

\begin{remark}
In the above definition we do not  assume that $z \in L^{p}(0,T;\mx).$ In fact, if $T < \infty$ or $0 \in \rho(A),$ where $\rho(A)$ is the resolvent set of $A,$ $z' \in L^{p}(0,T;\mx)$ can be replaced by $z \in W^{1,p}(0,T;\mx)$ (\cite[Theorem 2.4]{Dor1993}).
\end{remark}

We now state our first main result:
\begin{theorem} \label{th:main1}
Let  $1< p,q < \infty$ and $T < \infty.$ Then $\mathbb{A} \in \mathcal{MR}_{p}([0,T]; \mathbb{H}^{q}(\Omega)).$ In particular,  the operator $\mathbb{A}$ generates an analytic semigroup on  $\mathbb{H}^{q}(\Omega)$
for any $q\in (1,\infty).$
\end{theorem}

To prove the above result we use  the  characterization of maximal $L^{p}$ regularity  due to Weis (\cite[Theorem 4.2]{Weis01}), which says that maximal $L^{p}$ regularity property in a UMD Banach space, in particular for $L^{q}$ spaces,  is equivalent to the $\mr$-sectoriality property of the operator (see \cref{sec:R} for definition and properties of $\mr$-sectorial operators).

The maximal $L^{p}$-$L^{q}$ regularity property in finite time interval of the system \eqref{fsi-l}, when $\Omega = \mathbb{R}^{3},$ was already proved in  \cite{Gei13,Got12}.  However, the approach of those papers is different from our approach. In fact, in those papers, fluid and structure equations are treated separately and maximal $L^{p}$-$L^{q}$ regularity property of the linear system \eqref{fsi-l} is proved by a fixed point argument.  In our approach, we solve the fluid and structure equations simultaneously. In the study of fluid-structure interactions, this method is known as a monolithic approach. We refer to Maity and Tucsnak \cite{MT17} where a similar approach has been used to prove maximal $L^{p}$-$L^{q}$ regularity for several other fluid structure models.

The main advantage of such approach is that, by studying resolvent of the  linear operator $\mathbb{A},$ we can conclude that the operator $\mathbb{A}$ generates a $C^{0}$-semigroup of negative type. This allows us to obtain the  maximal $L^{p}$-$L^{q}$ regularity of the  system  \eqref{fsi-l} on $[0,\infty).$  As a consequence, we obtain global existence and uniqueness  for the full non-linear system \eqref{eq:mainsys} under a smallness condition on the initial data.

In order to state global existence and uniqueness result, we introduce some notation. Firstly
$W^{s,q}(\Omega)$, with $s\geqslant 0$ and $q>1$, denote the usual Sobolev spaces.
We introduce the space
\[ L^{q}_{m}(\Omega) = \left\{ f \in L^{q}(\Omega) \mid  \int_{\Omega} f  = 0 \right\} \]
and we set
\[ W^{s,q}_{m}(\Omega) = W^{s,q}(\Omega) \cap L^{q}_{m}(\Omega).  \]
 Let $k \in \mathbb{N}$.
For every $0 < s < k, $ $1\leqslant p < \infty$, $1\leqslant q < \infty$, we define the Besov spaces by real interpolation of Sobolev spaces
\begin{align*}
B^{s}_{q,p}(\Omega) = (L^{q}(\Omega), W^{k,q}(\Omega))_{s/k,p}.
\end{align*}
We refer to \cite{Triebel-2} for more details on Besov spaces.  We also need a definition of Sobolev spaces in the time dependent domain $\Omega_{F}(t)$. Let $\Lambda(t,\cdot)$ be a $C^{1}$-diffeomorphism from $\ofo$ onto $\oft$ such that all the derivatives up to second order in space variable and all the derivatives up to first order in time variable exist are continuous. For all functions $v(t,\cdot) : \oft \mapsto \mathbb{R},$ we denote $\widehat v(t,y) = v(t, \Lambda(t,y)).$ Then for any $1 < p,q < \infty$  we define
\begin{align*}
&L^{p}(0,T;L^{q}(\Omega_{F}(\cdot))) = \left\{ v  \mid \widehat v \in L^{p}(0,T;L^{q}(\ofo))\right\}, \\
&L^{p}(0,T;W^{2,q}(\Omega_{F}(\cdot))) = \left\{ v  \mid \widehat v \in L^{p}(0,T;W^{2,q}(\ofo))\right\}, \\
&W^{1,p}(0,T;L^{q}(\Omega_{F}(\cdot))) = \left\{ v  \mid \widehat v \in W^{1,p}(0,T;L^{q}(\ofo))\right\}, \\
&C([0,T];B^{2(1-1/p)}_{q,p}(\Omega_{F}(\cdot))) = \left\{ v  \mid \widehat v \in C([0,T];B^{2(1-1/p)}_{q,p}(\ofo)\right\}.
\end{align*}

\begin{theorem} \label{th-g}
Let $1 < p,q < \infty$ satisfying the conditions $\ds \frac{1}{p} + \frac{1}{2q} \neq 1$ and $\ds \frac{1}{p} + \frac{3}{2q} \leqslant \frac{3}{2}.$ Let $\eta \in (0,\eta_{0}),$ where $\eta_{0}$ is the constant introduced in \cref{th:exp-st}.
Then there exist two constants $\delta_{0} > 0$ and $C> 0,$ depending on $p,q,\eta$ and $\ofo,$ such that, for all  $\delta \in (0,\delta_{0})$ and for all   $(u_{0},\ell_{0},\omega_{0}) \in  B^{2(1-1/p)}_{q,p}(\ofo)^{3} \times \rt \times \rt$ satisfying the compatibility conditions
\begin{align*}
&\mathrm{div} \  u_{0} = 0 \mbox{ in } \ofo,  \\
& u_{0} = \ell_{0} + \omega_{0} \times y \mbox{ on } \poso, \quad u_{0} = 0 \mbox{ on } \partial\Omega \mbox{ if } \frac{1}{p} + \frac{1}{2q} < 1  \\
\mbox{ and } & u_{0} \cdot n = (\ell_{0} + \omega_{0} \times y) \cdot n \mbox{ on } \poso, \quad \quad u_{0} \cdot n = 0 \mbox{ on } \partial \Omega \mbox{ if } \frac{1}{p} + \frac{1}{2q} > 1,
\end{align*}
and
\begin{align}
\|u_{0}\|_{B^{2(1-1/p)}_{q,p}(\ofo)^{3}} + \|\ell_{0}\|_{\rt} + \|\omega_{0}\|_{\rt} \leqslant \delta,
\end{align}
the system \eqref{eq:mainsys} admits a unique strong solution $(u,\pi,\ell, \omega)$ in the class of functions satisfying
\begin{multline} \label{es:main}
\|e^{\eta (\cdot)}u \|_{L^{p}(0,\infty;W^{2,q}(\Omega_{F}(\cdot)))^{3}} + \|e^{\eta (\cdot)}u \|_{W^{1,p}(0,\infty;L^{q}(\Omega_{F}(\cdot)))^{3}} + \|e^{\eta (\cdot)}u\|_{L^{\infty}(0,\infty;B^{2(1-1/p)^{3}}_{q,p}(\Omega_{F}(\cdot)))}   \\
+ \|e^{\eta (\cdot)} \pi \|_{L^{p}(0,\infty;W_{m}^{1,q}(\Omega_{F}(\cdot)))}
 +\|a\|_{L^{\infty}(0,\infty;\rt)} + + \|e^{\eta (\cdot)} a' \|_{L^{p}(0,\infty;\rt)} \\  + \|e^{\eta (\cdot)} a'' \|_{L^{p}(0,\infty;\rt)} + \|e^{\eta (\cdot)}\omega \|_{W^{1,p}(0,\infty;\rt)} \leqslant C\delta.
\end{multline}
Moreover, $\mathrm{dist} \left(\ost,\partial\Omega\right) \geqslant \alpha/2  $ for all $t \in [0,\infty).$
In particular, we have
\begin{equation*}
\|u(t,\cdot)\|_{B^{2(1-1/p)}_{q,p}(\oft)} + \|a'(t)\|_{\rt} + \|\omega(t)\|_{\rt} \leqslant C \delta e^{-\eta t}.
\end{equation*}
\end{theorem}

\begin{remark}
When $p =q =2,$ the above result was proved in \cite[Corollary 9.2]{TT2}.
\end{remark}

\begin{remark}
Our proof of \cref{th-g} also applies to the $2$ dimensional case.  In this case, we have to  choose $1 < p,q < \infty$ such that $\ds \frac{1}{p} + \frac{1}{2q} \neq 1$ and $\ds \frac{1}{p} + \frac{1}{2q} \leqslant \frac{3}{2}.$
\end{remark}

The plan of this paper is as follows.  In \cref{sec:R}, we recall the definition and some basic properties  of $\mr$-sectorial operators.  In \cref{sec:RA}, we prove \cref{th:main1}. The stability of the operator $\mathbb{A}$ is proved in \cref{sec:exp}. Maximal $L^{p}$-$L^{q}$ regularity of the linear fluid structure system on $(0,\infty)$
is studied in \cref{sec:maxlp}. Finally, in \cref{sec:gr}  we prove \cref{th-g}.

%%%%%%%%%%%%%%%%%%%%%%%%%%%%%%%%%%%%%%%%%%%%%%%%%%%%%%%%%%%%%%%%%%%%%%%%%%%%%%%%%%%%%%%%%%%%%%%%%%%%%%%%%%%%%%%%%%%%%%%

\section{Some background on $\mr$-sectorial operators} \label{sec:R}
In this section we recall some definitions and  basic results concerning maximal regularity and  $\mathcal{R}$-boundedness in Banach spaces. For detailed information on these subjects we refer to \cite{Weis01,KW04,DenkHieberPruss} and  references therein.  For $\theta \in (0,\pi)$ we define  the sector $\Sigma_{\theta}$ in the complex plane by
\begin{align}
\Sigma_{\theta} =  \{ \lambda \in \mathbb{C} \setminus \{0\} \ \ \mid \ \  |\mathrm{arg} \lambda| < \theta \}.
\end{align}

In order to state  Weis' theorem concerning maximal $L^{p}$-regularity of the Cauchy problem \eqref{eq:ab-cauchy} we need to introduce so-called $UMD$ spaces.

\begin{definition}
Let $\mx$ be a Banach space. The Hilbert transform of a function $f \in {\mathcal S}(\mathbb{R};\mx),$ the Schwartz space of $\mx$-valued rapidly decreasing functions, is defined by
\begin{align*}
Hf(t) = \frac{1}{\pi} \lim_{\epsilon \mapsto 0} \int_{|s| > \epsilon} \frac{f(t-s)}{s} \ ds, \quad t \in \mathbb{R}.
\end{align*}
A Banach space  $\mx$ is said to be of class ${\mathcal {HT}}$, if the Hilbert transform is bounded on $L^{p}(\mathbb{R};\mx)$ for some (thus all) $1 < p < \infty.$
\end{definition}

These spaces are also called $UMD$ Banach spaces, where $UMD$ stands for \textit{unconditional martingale differences}.  Hilbert spaces, all closed subspaces and quotient spaces of $L^{q}(\Omega)$ with $1 < q  < \infty$ are examples of $UMD$ spaces. We refer the reader to \cite[pp. 141-147]{Amann} for more information about $UMD$ spaces.

We next introduce the notion of $\mathcal{R}$-bounded  family of operators and $\mathcal{R}$-sectoriality of a densely defined linear operator.
\begin{definition} [${\mathcal R}$ - bounded family of operators]
Let $\mx$ and $\my$ be  Banach spaces. A family of operators $\mathcal{T} \subset \mathcal{L}(\mx,\my)$ is called $\mathcal{R}$- bounded on $\mathcal{L}(\mx,\my),$ if there exist constants $C >0$ and $p \in [1,\infty)$ such that for every $n \in \mathbb{N},$ $\{T_{j}\}_{j=1}^{n} \subset \mathcal{T},$ $\{x_{j}\}_{j=1}^{n} \subset \mx$  and for all sequences $\{r_{j}(\cdot)\}_{j=1}^{n}$ of independent, symmetric, $\{ -1,1\}$ valued random variables on $[0,1]$, we have
\begin{align*}
\left\|\sum_{j=1}^{n}r_{j}(\cdot) T_{j} x_{j}\right\|_{L^{p}([0,1];\my)} \leqslant C \left\|\sum_{j=1}^{n}r_{j}(\cdot)  x_{j}\right\|_{L^{p}([0,1];\mx)}.
\end{align*}
The smallest such $C$ is called ${\mathcal R}$-bound of $\mathcal{T}$ on $\mathcal{L}(\mx,\my)$ and denoted by $\mathcal{R}_{\mathcal{L}(\mx,\my)}(\mathcal{T}).$
\end{definition}

\begin{definition}[$\mathcal{R}$-sectorial operator]
Let $A$ be a densely defined closed linear operator on a Banach space $\mx$ with domain $\mathcal{D}(A).$ Then $A$ is said sectorial of angle $\theta \in (0,\pi)$ if $\sigma(A) \subseteq \overline{\Sigma_{\theta}}$ and  for any $\theta_{1} > \theta$ the set $\ds \left\{\lambda(\lambda I -  A)^{-1} \ \ \mid \ \ \theta_{1} \leqslant |\mathrm{arg}(\lambda)| \leqslant \pi \right\}$ is bounded. If this set is $\mr$-bounded then $A$ is $\mr$-sectorial of angle $\theta.$
\end{definition}

We now state several useful properties concerning $\mathcal{R}$-boundedness, which will be used later on

\begin{proposition} $ $ \label{rem:Rbd}
\begin{itemize}
\item[(1)] If $\mathcal{T} \subset \mathcal{L}(\mx,\my)$ is $\mathcal{R}$-bounded then it is uniformly bounded with
\begin{align*}
\mathrm{sup} \left\{ \|T\| \mid T \in \mathcal{T} \right\} \leqslant \mathcal{R}_{\mathcal{L}(\mx,\my)}(\mathcal{T}).
\end{align*}
\item[(2)] If $\mx$ and $\my$ are Hilbert spaces, $\mathcal{T} \subset \mathcal{L}(\mx,\my)$ is $\mr$-bounded if and only if $\mathcal{T}$  is uniformly bounded.
\item[(3)] Let $\mx$ and $\my$ be Banach spaces and let $\mathcal{T}$ and $\mathcal{S}$  be $\mathcal{R}$-bounded families on $\mathcal{L}(\mx,\my).$ Then $\mathcal{T} + \mathcal{S}$ is also $\mathcal{R}$-bounded  on $\mathcal{L}(\mx,\my),$ and
\begin{align*}
{\mathcal R}_{\mathcal{L}(\mx,\my)}(\mathcal{T} + \mathcal{S}) \leqslant {\mathcal R}_{\mathcal{L}(\mx,\my)}(\mathcal{T} ) +  {\mathcal R}_{\mathcal{L}(\mx,\my)}( \mathcal{S}).
\end{align*}
\item[(4)]  Let $\mx,\my$ and $\mz$ be Banach spaces and let $\mathcal{T}$ and $\mathcal{S}$  be $\mathcal{R}$-bounded families on $\mathcal{L}(\mx,\my)$ and $\mathcal{L}(\my,\mz)$ respectively. Then $\mathcal{S}\mathcal{T}$ is  $\mathcal{R}$-bounded  on $\mathcal{L}(\mx,\mz),$ and
\begin{align*}
{\mathcal R}_{\mathcal{L}(\mx,\mz)}(\mathcal{S} \mathcal{T} ) \leqslant {\mathcal R}_{\mathcal{L}(\mx,\my)}(\mathcal{T} )   {\mathcal R}_{\mathcal{L}(\my,\mz)}( \mathcal{S}).
\end{align*}
\end{itemize}
\end{proposition}

The following characterization of maximal $L^{p}$ regularity is due to Weis (\cite[Theorem 4.2]{Weis01})
\begin{theorem} \label{thm:Weis}
Let $\mx$ be a Banach space of class ${\mathcal{HT}},$ $1 < p < \infty$ and let $A$ be a closed, densely defined unbounded operator with domain $\md(A).$
Then $A$ has maximal $L^{p}$-regularity on $\mathbb{R}^{+}$ if and only if
\begin{equation}\label{give_number}
\mr_{\mathcal{L}(\mx)}\left\{\lambda (\lambda - A)^{-1} \mid \lambda  \in \Sigma_{\theta}\right\} \leqslant C  \mbox{ for some } \theta > \pi/2.
\end{equation}
In other words, $A$ has maximal $L^{p}$-regularity if and only if $-A$ is $\mr$-sectorial of angle $\theta < \pi/2.$
\end{theorem}

Next we state  a  perturbation result  due to Kunstmann and Weis \cite{WeisKun01}, which states that $\mr$-boundedness is preserved by $A$ small perturbations.
\begin{proposition} \label{pr:perturb}
Let $\mx$ be a Banach space  let $A$ be a closed, densely defined unbounded operator with domain $\md(A).$ Let  us assume that there exist $\gamma_{0} \geqslant  0$ and $\theta \in (0,\pi)$ such that
\begin{equation*}
\mr_{\mathcal{L}(\mx)}\left\{\lambda (\lambda - A)^{-1} \mid \lambda  \in \gamma_{0} + \Sigma_{\theta}\right\} \leqslant C.
\end{equation*}
 Let $B$ be a $A$-bounded operator with relative bound zero, i.e., for all $\delta > 0$ there exists $C(\delta) > 0$ such that
\begin{align}
\|Bz\| \leqslant \delta \|Az\| + C(\delta) \|z\| \mbox{ for all }   z \in \md(A).
\end{align}
Then there there exists  $\mu_{0} \geqslant \gamma_{0}$ such that
\begin{equation*}
\mr_{\mathcal{L}(\mx)}\left\{\lambda (\lambda - (A+B))^{-1} \mid \lambda  \in \mu_{0} + \Sigma_{\theta}\right\} \leqslant C.
\end{equation*}
\end{proposition}

We conclude this section by stating an existence and uniqueness result for the abstract Cauchy problem \eqref{eq:ab-cauchy} on $\mathbb{R}^{+},$ which we will use to prove maximal $L^{p}$-$L^{q}$ regularity of the system  \eqref{fsi-l} (see \cite[Theorem 2.4]{Dor1993}).

\begin{theorem} \label{th:ab-i}
Let $\mx$ be a Banach space of class $\mathcal{HT},$ $1 < p < \infty$ and let $A$ be a closed, densely defined unbounded operator with domain $\md(A).$  Let  us assume that $A \in \mathcal{MR}_{p}([0,T];\mx)$
and  the semigroup generated by $A$ has negative exponential type. Then for every $z_{0} \in (\mx, \md(A))_{1-1/p,p}$ and for every $f \in L^{p}(0,\infty;\mx),$ \eqref{eq:ab-cauchy} admits a unique strong solution in $L^{p}(0,\infty;\md(A)) \cap W^{1,p}(0,\infty;\mx).$ Moreover, there exists a positive constant $C$ such that
\begin{align} \label{est:ab}
\|z\|_{L^{p}(0,\infty;\md(A))} + \|z\|_{W^{1,p}(0,\infty;\mx)} \leqslant C \left( \|z_{0}\|_{(\mx, \md(A))_{1-1/p,p}} + \|f\|_{L^{p}(0,\infty;\mx)}\right).
\end{align}
\end{theorem}

%%%%%%%%%%%%%%%%%%%%%%%%%%%%%%%%%%%%%%%%%%%%%%%%%%%%%%%%%
\section{$\mr$-sectoriality of the operator $\mathbb{A}.$} \label{sec:RA}
Let us recall the operator $\mathbb{A}$ introduced in \eqref{sg-TT}. The aim of this section is to prove \cref{th:main1}.  Due to \cref{thm:Weis}, it is enough to prove the following theorem:

\begin{theorem} \label{th:31}
Let $1 < q < \infty.$ There exists $\mu_{0} > 0$ and $\theta \in (\pi/2,\pi)$ such that $\mu_{0} + \Sigma_{\theta}\subset \rho(\mathbb{A})$ and
\begin{equation}
\mr_{\mathcal{L}(\mathbb{H}^{q}(\Omega))}  \left\{ \lambda (\lambda I - \mathbb{A})^{-1} \mid \lambda \in \mu_{0} + \Sigma_{\theta} \right\}  \leqslant C.
\end{equation}
\end{theorem}
Let us remark that when $q =2$ the above theorem is already proved in \cite{TT1}. To prove the above theorem we will first obtain an equivalence  formulation of the resolvent equation.

\subsection{Reformulation of the resolvent equation}
Given $\lambda\in \mathbb{C}$,  $f \in L^{q}(\ofo)^{3}$ and $(g_{1},g_{2}) \in \ct \times \ct,$ we consider the system
\begin{alignat}{2} \label{eq:r1}
&\lambda u - \nu \Delta u + \nabla \pi  = f, \quad \mathrm{div} \ u = 0  &  \mbox{ in }   \ofo, \notag \\
& u = 0   & \mbox{ on }  \partial\Omega,\notag  \\
&  u  = \ell  + \omega \times y &   y \in   \poso, \notag \\
&  \lambda  m \ell =  -  \int_{\partial \Omega_{S}(0)}  \sigma(u,\pi) n \ d\gamma  +  g_{1}, \\
&  \lambda J(0) \omega =  - \int_{\partial \Omega_{S}(0)} y \times   \sigma(u,\pi)n \ d\gamma + g_{2}, \notag
\end{alignat}
of unknowns $(u,\pi,\ell,\omega)$.
 Following \cite{TT2, TT1} we have the following equivalence

\begin{proposition} \label{prop:ev1}
Let $1 < q < \infty.$ Let us assume that $f \in L^{q}(\ofo)$ and $(g_{1},g_{2}) \in \ct \times \ct.$ Then $(u,\pi,\ell,\omega) \in W^{2,q}(\ofo)^{3} \times W^{1,q}_{m}(\ofo) \times \ct \times \ct$ is a solution to \eqref{eq:r1} if and only if
\begin{align}
(\lambda I - \mathbb{A}) v = \mathbb{P} F
\end{align}
where
\begin{align*}
v = u \mathbbm{1}_{\ofo} + (\ell + \omega \times y) \mathbbm{1}_{\oso}, \quad F = \mathbb{P} \left( f\mathbbm{1}_{\ofo} + \left(m^{-1} g_{1} + J(0)^{-1} y \times g_{2} \right)\mathbbm{1}_{\oso} \right).
\end{align*}
\end{proposition}

Next, we derive another equivalent formulation of the resolvent equation \eqref{eq:r1}. In this case, we do not extend the fluid velocity by the structure velocity everywhere in the domain $\Omega,$ rather we work on the fluid domain $\ofo.$    The idea is to eliminate the pressure from both the fluid and the structure equations. To eliminate the pressure from the fluid equation we use Leray projector $$\mpp : L^{q}(\ofo)^{3} \mapsto \mathcal{V}^q_{n} (\ofo) := \left \{\varphi \in { L}^q(\ofo)^{3} \mid \mathrm{div} \ \varphi = 0, \ \varphi \cdot { n} = 0 \mbox{ on } \partial\ofo\right\}.$$
Note that the projector $\mpp$ is different from the projector $\mathbb{P}$ used in \eqref{sg-TT}.
Following  \cite{Raymond}, first, we  decompose  the fluid velocity into two parts  $\mathcal{P}u$ and $(I-\mpp)u.$ Next, we obtain an expression of pressure, which can be broken down into two parts, one which depends on $\mpp u$ and another part which depends on $(\ell, \omega).$  This will allow us to eliminate the pressure term from the structure equations  and  rewrite the system \eqref{eq:r1} as an operator equation of $(\mpp u, \ell, \omega).$

 The advantage of this formulation is that we can prove the $\mr$-boundedness of the resolvent operator just by using the fact that Stokes operator with {\em homogeneous Dirichlet boundary conditions} is $\mr$-sectorial and a perturbation argument.  This idea has been used in several fluid-solid interaction problems in the Hilbert space setting and when the structure is deformable and located at the boundary (see, for instance, \cite{Raymond-fsi,  MR17} and references therein).

Let $1 < q < \infty$ and $q'$ denote the conjugate of $q,$ i.e., $\ds \frac{1}{q} + \frac{1}{q'}=1.$ Let $ n$ denote the  normal to $\partial \ofo$ exterior to $\ofo.$  For $1< q< \infty,$ we first introduce the  space
\begin{align*}
W_{q, div } (\Omega) = \left\{ \varphi \in L^{q}(\Omega)^{3} \mid \mathrm{div} \ \varphi \in L^{q}(\Omega) \right\},
\end{align*}
equipped with the norm
\begin{align*}
\|\varphi\|_{W_{q, div } (\Omega)} := \|\varphi\|_{L^{q}(\Omega)^{3}} + \|\mathrm{div} \ \varphi \|_{L^{q}(\Omega)}.
\end{align*}
It is easy to check that $W_{q, div } (\Omega)$ is a Banach space. We have the following classical lemma:
\begin{lemma} \label{Normaltrace}
Let $\Omega$ be a bounded domain with smooth boundary. The linear mapping
\begin{align*}
\varphi \mapsto \gamma_{n} \varphi := \varphi|_{\partial\Omega}  \cdot n
\end{align*}
defined on $C^{\infty}(\overline{\Omega})^{3}$ can be extended to a continuous and surjective map from $W_{q, div } (\Omega)$ onto $W^{-1/q,q}(\partial\Omega).$
\end{lemma}
\begin{proof}
For proof see \cite[Lemma 1]{Fujiwara}.
\end{proof}

Let us set
\begin{equation*}
 \mathcal{V}^q_{n} (\ofo)  = \overline{\left\{ \varphi \in C^{\infty}_{c}(\ofo) \mid \mathrm{div} \ \varphi = 0 \right\}}^{\|\cdot\|_{L^{q}}}.
\end{equation*}
As $\ofo$ is bounded, we actually have
\begin{equation*}
\mathcal{V}^q_{n} (\ofo)= \left \{\varphi \in { L}^q(\ofo)^{3} \mid \mathrm{div} \ \varphi = 0, \ \varphi \cdot { n} = 0 \mbox{ on }
 \partial\ofo\right\}.
\end{equation*}
We  have the following Helmholtz-Weyl decomposition of $L^{q}(\ofo)^{3}$ 
\begin{proposition} \label{hw}
The space ${ L}^q(\ofo)^{3}$ admits the following decomposition in a direct sum:
\begin{align*}
  {L}^q(\ofo)^{3} = \mv^q_{n}(\ofo) \oplus G^{q}(\ofo),
\end{align*}
where
\begin{equation*}
G^{q}(\ofo) =  \left\{ \nabla \varphi \mid \varphi \in W^{1,q}(\ofo)\right\}.
\end{equation*}
The  projection operator from ${ L}^q(\ofo)^{3}$ onto $\mv^q_{n}(\ofo)$ is denoted by $\mpp.$  The projector
$\mpp:{ L}^q(\ofo)^{3} \mapsto \mv^q_{n}(\ofo)$ is defined by
\begin{align*}
 \mpp u  = u - \nabla \varphi,
\end{align*}
where $\varphi \in W^{1,q}(\ofo)$ solves the following Neumann problem
\begin{align*}
 \Delta \varphi = \mathrm{div} \ u  \mbox{ in } \ofo, \quad \frac{\partial \varphi}{\partial n} = \varphi \cdot n \mbox{ on } \partial\ofo.
\end{align*}
\end{proposition}
\begin{proof}
For the proof of the above result we refer to Section 3 and Theorem 2 of 
\cite{Fujiwara}.
\end{proof}
Let us denote by $A_{0} = \nu \mpp \Delta,$ the Stokes operator in $\mv^{q}_{n}(\ofo)$ with domain
\begin{align*}
\mathcal{D}(A_{0}) =  W^{2,q}(\ofo) \cap W^{1,q}_{0}(\ofo)^{3} \cap \mathcal{V}^q_{n} (\ofo).
\end{align*}
\begin{proposition} \label{prop:rsecS}
The Stokes operator $-A_{0}$ is ${\mathcal R}$-sectorial in $\mv^{q}_{n}(\ofo)$ of angle $0.$  In particular, there exists $\theta \in (\pi/2, \pi)$ such that
\begin{align}
\mr_{\mathcal{L}(\mv^{q}_{n}(\ofo))} \left\{ \lambda(\lambda I - A_{0})^{-1} \mid \lambda \in  \Sigma_{\theta} \right\} \leqslant C.
\end{align}
Moreover, $A_{0}$ generates a $C^{0}$-semigroup of negative type.
\end{proposition}

\begin{proof}
For proof we refer to Theorem 1.4 and Corollary 1.4 of \cite{Gei10}.
\end{proof}

Now we are going to rewrite the first three equations of \eqref{eq:r1} in terms of $\mpp u$ and $(I - \mpp)u.$ Let us consider the  following problem
\begin{align} \label{Ds}
\begin{cases}
- \nu \Delta w + \nabla \psi = 0, \quad \mathrm{div} \ w = 0, \; \quad y \in \ofo,  \\
 w = 0, \;\qquad\qquad \qquad\qquad y \in \partial\Omega, \\
   w = \ell + \omega \times y, \; \qquad \quad  y \in  \partial\oso, \\
 \ds \int_{\ofo} \psi \ {\rm d}y = 0.
\end{cases}
\end{align}
\begin{lemma} \label{lem:lift}
Let $(\ell, \omega) \in \ct \times \ct$ and let $\{e_{i}\}$ denote the canonical basis in $\ct$. Then the solution $(w,\pi)$ of \eqref{Ds} can be expressed as follows
 \begin{align}
 w = \sum_{i=1}^{3} \ell_{i} W_{i} + \sum_{i=4}^{6} \omega_{i-3} W_{i},  \quad \psi = \sum_{i=1}^{3} \ell_{i} \Psi_{i} + \sum_{i=4}^{6} \omega_{i-3} \Psi_{i}, \quad
\end{align}
where  $(W_{i},\Psi_{i})$, $i = 1, 2, \cdots, 6$ solves the following system
\begin{align} \label{eq:UP}
\begin{cases}
-\nu \Delta W_{i} +  \nabla \Psi_{i} = 0, \quad \mathrm{div} W_{i} = 0, \quad y\in  \ofo,  \\
 W_{i} = 0, \qquad  \qquad y \in  \partial\Omega, \\
 W_{i} = e_{i}, \mbox{ for } i=1,2,3 \mbox{ and }  W_{i} = e_{i-3} \times y, \mbox{ for } i=4,5,6, \quad y\in \partial\ofo, \\
\ds  \ds \int_{\ofo} \Psi \ {\rm d}y = 0.
\end{cases}
\end{align}
Moreover,
\begin{align*}
\begin{pmatrix}
\displaystyle
\int_{\partial\oso} \sigma(w,\psi)n \  d\gamma \\ \displaystyle \int_{\partial\oso} y \times \sigma (w, \psi)n \  d\gamma
\end{pmatrix} = \mathbb{B} \begin{pmatrix}
\ell \\ \omega
\end{pmatrix},
\end{align*}
where
\begin{align} \label{matM1}
\displaystyle \mathbb {B}_{i,j} = \int_{\ofo} DW_{i} : DW_{j}.
\end{align}
and the matrix $\mathbb{B}$ is invertible.
\end{lemma}

\begin{proof}
The expressions of $w$ and $\pi$ are easy to see. The expression of the matrix $\mathbb{B}$ follows easily by putting the expressions of $w$ and $\pi$ and by integration by parts. The proof of invertibility of the matrix $\mathbb{B}$ can be found in   \cite[Chapter 5]{HB65}.
\end{proof}

Let us introduce the following operators:
\begin{itemize}
\item The Dirichlet lifting operator $D \in {\mathcal L}(\mathbb{C}^{6},W^{2,q}(\ofo))$ and
$D_{pr} \in {\mathcal L}(\mathbb{C}^{6},W^{1,q}_{m}(\ofo))$ defined by
\begin{align}
 D(\ell,\omega) = w, \quad D_{pr}(\ell,\omega) = \psi,
\end{align}
where $(w,\psi)$  is the solution to the problem \eqref{Ds}.
\item The Neumann operator $N \in {\mathcal L}(W^{1-1/q,q}_{m}(\partial\ofo), W^{2,q}_{m}(\ofo))$ defined by $N h = \varphi,$ where $\varphi$ is the solution to the Neumann problem
\begin{align} \label{N}
\Delta\varphi= 0 \mbox{ in } \ofo, \quad \frac{\partial \varphi}{\partial n} = h \mbox{ on } \partial\ofo.
\end{align}
We set
\begin{align} \label{Ns}
N_{S} h = N (\mathbbm{1}_{\partial \oso}  h) \mbox{ for }  h \in W^{1-1/q,q}_{m}(\partial\oso).
\end{align}
\end{itemize}

We now rewrite the equations satisfied by $u$ in system \eqref{eq:r1} as a new system of two equations, one satisfied by $\mpp u$ and another by $(I-\mpp) u.$ More precisely, we have the following proposition
\begin{proposition} \label{equiv1}
Let $1 <  q < \infty.$ Let us assume that  $(f, \ell,\omega) \in \mv^q_{n}(\ofo) \times \ct \times \ct.$ A pair $(u,\pi) \in W^{2,q}(\ofo) \times W^{1,q}_{m}(\ofo)$  satisfies the system
\begin{align} \label{stokes1}
\begin{cases}
\lambda u - \nu \Delta u + \nabla \pi = f, \quad \mathrm{ div} \ u = 0, \quad y \in \ofo,  \\
 u  = 0 \quad  y \in \partial \Omega,  \\
 u = \ell + \omega \times y \quad  y \in  \partial\oso,
\end{cases}
\end{align}
if and only if
\begin{align} \label{stokes2}
\begin{cases}
\lambda \mpp u  - A_{0} \mpp u  + A_{0} \mpp D(\ell,\omega) = \mpp f,  \\
(I - \mpp) u = (I - \mpp) D (\ell,\omega )  \\
\pi = N (\nu \Delta \mpp u \cdot n) - \lambda N_{S} ((\ell + \omega \times y) \cdot n) .
\end{cases}
\end{align}
\end{proposition}

\begin{proof}
Let $(u,\pi) \in W^{2,q}(\ofo)^{3} \times W^{1,q}_{m}(\ofo)$ satisfies the system \eqref{stokes1}. We set
\begin{equation*}
\widetilde u  = u - D(\ell, \omega), \quad \widetilde \pi = \pi - D_{pr}(\ell, \omega).
\end{equation*}
The pair $(\widetilde u, \widetilde p)$ satisfies the following system
\begin{align*}
&\lambda \widetilde u + \lambda D(\ell, \omega) - \nu \Delta \widetilde u + \nabla p = f, \quad \mathrm{ div} \ \widetilde u = 0, \mbox{ in } \ofo, \notag \\
&\widetilde u  = 0 \quad   \mbox{ on } \partial\ofo.
\end{align*}
Note that $\widetilde u \in \mathcal{D}(A_{0})$ and $\mpp \widetilde u = \widetilde u.$  Thus applying the projection $\mpp$ on the above system it is easy to see that $\mpp  u$ satisfies the following
\begin{align*}
\lambda \mpp u   - A_{0} \mpp u  + A_{0} \mpp D(\ell,\omega) = \mpp f.
\end{align*}
Since $(I - \mpp) \widetilde u = 0,$ we obtain
\begin{align*}
(I - \mpp) u = (I -\mpp)(\widetilde u + D(\ell, \omega)) = (I - \mpp) D (\ell,\omega ).
\end{align*}
Note that, from the expression of $\mpp$ in \cref{hw}, it follows that $\Delta(I - \mpp) u = 0$ in $\ofo.$ Therefore  the first equation of \eqref{stokes1} can be written as 
\begin{equation*}
\lambda u - \Delta \mpp u + \nabla \pi = f  \quad \mbox{ in } \ofo.
\end{equation*}
By applying the divergence and normal trace operators to the above equation, we obtain that $\pi$ is the solution of the problem
\begin{align} \label{pi1}
\begin{cases}
\Delta \pi =0  \qquad \mbox{ in } \ofo,   \\
\ds \frac{\partial \pi}{\partial n} =  \nu \Delta \mpp u \cdot n - \lambda u \cdot n \qquad \mbox{ on } \partial \ofo.
\end{cases}
\end{align}
Since $\mathrm{div} \Delta \mpp u = 0$ it follows that $\nu \Delta \mpp u \cdot n$  belongs to $W^{-1/q,q}(\partial\Omega)$  and satisfies the following condition
\begin{align*}
\left\langle \nu \Delta \mpp u \cdot n , 1 \right\rangle_{W^{-1/q,q}, W^{1-1/q',q'}}  = 0
\end{align*}
Since $\ofo$ is a smooth domain, $\eqref{pi1}$ admits a unique solution in $W^{1,q}_{m}(\ofo)$ (\cite[Theorem 9.2]{Fab98}) and the expression of $\pi$ in \eqref{stokes2} follows from the definition of the operators $N$ and $N_{S}.$

Conversely,  let  $(u,\pi) \in W^{2,q}(\ofo)^{3} \times W^{1,q}_{m}(\ofo)$ satisfies the system \eqref{stokes2}. Since $(I - \mpp)u = (I - \mpp) D(\ell, \omega)$ we get $\widetilde u : = u - D(\ell, \omega) \in \mathcal{D}(A_{0}).$ Thus $\eqref{stokes2}_{1}$ can be written as
\begin{align*}
\mpp \left( \lambda \widetilde u - A_{0} \widetilde u \right) = \mpp ( f - \lambda D(\ell,\omega)).
\end{align*}
Therefore, there exists $\widetilde \pi \in W^{1,q}_{m}(\ofo)$ such that $(\widetilde u, \widetilde \pi)$ satisfies
\begin{align*}
\lambda \widetilde u - \nu \Delta \widetilde u + \nabla \widetilde \pi = f - \lambda D(\ell, \omega), \quad \mathrm{div} \ \widetilde u = 0 \mbox{ in } \ofo,  \quad \widetilde u = 0 \mbox{ on } \partial \ofo.
\end{align*}
Then $(u, \pi),$ with $\pi = \widetilde \pi +D_{pr} (\ell, \omega),$ satisfies the system \eqref{stokes1}.
\end{proof}

Using the expression of the pressure $\pi$ obtained in \eqref{equiv1}, we rewrite the equations satisfied by $\ell$ and $\omega$ in \eqref{eq:r1} in the form
\begin{align*}
\lambda m \ell &= - 2\nu \int_{\poso} \varepsilon(u)n \ d\gamma + \int_{\partial \Omega_{S}(0)}  \pi n \ d\gamma  + g_{1} \\
& = -2\nu \left[\int_{\poso} \varepsilon(\mpp u) n \ d\gamma + \int_{\poso} \varepsilon((I - \mpp) D (\ell,\omega))n \ d\gamma\right] \\
& + \int_{\partial\oso} N (\Delta \mpp u \cdot n)  n \ d\gamma  - \lambda \int_{\poso} N_{S} ((\ell + \omega \times y) \cdot n) n \ \dg  + g_{1},
\end{align*}
and
\begin{align*}
J(0) \lambda \omega  & = -2\nu \left[\int_{\poso} y \times \varepsilon(\mpp u)n \ d\gamma + \int_{\poso} y \times \varepsilon((I - \mpp) D (\ell,\omega))n \ d\gamma\right] \\
& \quad \quad + \int_{\partial\oso} y \times  N (\Delta \mpp u \cdot n)  n \ d\gamma  -  \lambda \int_{\poso} y \times N_{S} ((\ell + \omega \times y) \cdot n) n \ \dg   + g_{2}.
\end{align*}
The above two equations can be written as
\begin{align}
\lambda \mathbb{K} \begin{pmatrix}
\ell \\ \omega
\end{pmatrix} =
\mathcal{C}_{1} \mpp u + \mathcal{C}_{2} \begin{pmatrix}
\ell \\ \omega
\end{pmatrix} +
\begin{pmatrix}
g_{1} \\ g_{2}
\end{pmatrix},
\end{align}
where
\begin{align} \label{eq:adm}
\mathbb{K}  = \begin{pmatrix}
m I_{3} & 0 \\ 0 & J(0)
\end{pmatrix} + \mathbb{M}, \quad
\mathbb{M} \begin{pmatrix}
\ell \\ \omega
\end{pmatrix} =  \begin{pmatrix}
\ds \int_{\poso} N_{S} ((\ell + \omega \times y) \cdot n) n \ \dg \\ \ds \int_{\poso} y \times N_{S} ((\ell + \omega \times y) \cdot n) n \ \dg
\end{pmatrix},
\end{align}
\begin{align}
\mathcal{C}_{1} \mpp u = \begin{pmatrix}
\ds -2\nu \int_{\poso} \varepsilon(\mpp u)n \ d\gamma + \int_{\partial\oso} N (\Delta \mpp u \cdot n) n \ d\gamma \\
\ds  -2\nu \int_{\poso} y \times \varepsilon(\mpp u) \ n \ d\gamma + \int_{\partial\oso}  y \times  N (\Delta \mpp u \cdot n)  n \ d\gamma
\end{pmatrix},
\end{align}
and
\begin{align}
\mathcal{C}_{2} \begin{pmatrix}
\ell \\ \omega
\end{pmatrix} =  \begin{pmatrix}
- \ds 2\nu \int_{\poso} \varepsilon((I - \mpp) D(\ell,\omega))n \ d\gamma \\
- \ds 2\nu \int_{\poso} y\times \varepsilon((I - \mpp) D(\ell,\omega))n \ d\gamma
\end{pmatrix}.
\end{align}

In the literature, the matrix $\mathbb{M}$ defined above  is known as the added mass operator. We are now going to show that the matrix $\mathbb{K}$  is an invertible matrix.

\begin{lemma} \label{lem:inv-adm}
The matrix $\mathbb{K}$ defined as in \eqref{eq:adm} is an invertible matrix.
\end{lemma}
\begin{proof}
The proof may be adapted from that of \cite[Lemma 4.6]{Gal02} (see also \cite[Lemma 4.3]{Gei13}).  Let us briefly explain the idea of the proof.  We are going to show that the matrix $\mathbb{M}$ is symmetric and semipositive definite. For that, we first derive an representation formula of the matrix $\mathbb{M}.$ Let us consider the following problem
\begin{align*}
&\Delta  \pi^{i} = 0 \mbox{ in } \ofo, \quad \frac{\partial  \pi^{i}}{\partial n} = 0 \mbox{ on } \partial\Omega, \quad \\
&\frac{\partial  \pi^{i}}{\partial n} = e_{i} \cdot n \mbox{ for }  i =1,2,3 \mbox{ and } \frac{\partial  \pi^{i}}{\partial n} = (e_{i-3} \times y) \cdot n \mbox{ for } i = 4,5,6, \quad y \in \partial \oso,
\end{align*}
where $\{e_{i}\}$ denote the canonical basis in $\mathbb{C}^{3}.$ Therefore, it is easy to see that
\begin{align*}
N_{S}((\ell + \omega \times y) \cdot n) = \sum_{i=1}^{3} \ell_{i} \pi^{i} + \sum_{i=4}^{6} \omega_{i-3} \pi^{i}.
\end{align*}
We define
\begin{align*}
m_{ij} =
\begin{cases}
\ds \int_{\partial\oso}  \pi^{i} n^{j} & \mbox{ for } 1 \leqslant i \leqslant 6, 1\leqslant  j \leqslant 3, \\
\ds \int_{\poso} \pi^{i}(e_{j-3} \times y) \cdot n & \mbox{ for } 1 \leqslant i \leqslant 6, 4\leqslant  j \leqslant 6.
\end{cases}
\end{align*}
One can easily check that, $\mathbb{M} = (m_{ij})_{1\leqslant i,j\leqslant 6}.$ With this representation and Gauss' theorem, we can  verify that $\mathbb{M}$ is symmetric and semipositive definite.
\end{proof}

Let us set
\begin{align} \label{eq:sp}
\mx  = \mathcal{V}^q_{n} (\ofo) \times \ct \times \ct.
\end{align}
and consider the operator $\mathcal{A}_{FS} : \mathcal{D}(\mathcal{A}_{FS}) \mapsto \mx$
defined by
\begin{align*}
\mathcal{D}(\mathcal{A}_{FS}) = \left\{ (\mpp u, \ell, \omega) \in \mx \mid \mpp u - \mpp D(\ell, \omega) \in \md(A_{0}) \right\},
\end{align*}
and
\begin{align} \label{op:afs}
\mathcal{A}_{FS} = \begin{pmatrix}
A_{0} & -A_{0}\mpp D \\  \mathbb{K}^{-1} \mathcal{C}_{1} & \mathbb{K}^{-1} \mathcal{C}_{2}
\end{pmatrix}.
\end{align}
Combining the above results, below we obtain an equivalence formulation of the system \eqref{eq:r1}.
\begin{proposition} \label{prop:ev2}
Let $1 <  q < \infty.$ Let us assume that  $(f, g_{1},g_{2}) \in \mathcal{V}^q_{n} (\ofo) \times \ct \times \ct.$ Then  $(u,p, \ell, \omega) \in W^{2,q}(\ofo) \times W^{1,q}_{m}(\ofo) \times \ct \times \ct$  satisfy the system \eqref{eq:r1} if and only if
\begin{align}
&(\lambda I - \mathcal{A}_{FS}) \begin{pmatrix}
\mpp u \\ \ell \\ \omega
\end{pmatrix} = \begin{pmatrix}
\mpp f \\ \widetilde g_{1} \\ \widetilde g_{2}
\end{pmatrix}, \\
&(I - \mpp) u = (I - \mpp) D (\ell,\omega ), \notag \\
&\pi = N (\nu \Delta \mpp u \cdot n) - \lambda N_{S} ((\ell + \omega \times y) \cdot n), \notag
\end{align}
where $(\widetilde g_{1}, \widetilde g_{2})^{\top}  = \mathbb{K}^{-1}(g_{1},g_{2})^{\top}.$
\end{proposition}

We end this subsection with the following lemma
\begin{lemma} \label{lem:eqiv-norm}
The map
\begin{align*}
(\mpp u, \ell, \omega) \mapsto \|\mpp u\|_{W^{2,q}(\ofo)} + \|\ell\|_{\ct} + \|\omega\|_{\ct},
\end{align*}
is a norm on $\mathcal{D}(\mathcal{A}_{FS})$  equivalent to the graph norm.
\end{lemma}
\begin{proof}
The proof is similar to that of \cite[Proposition 3.3]{Raymond-fsi}.
\end{proof}

\subsection{$\mathcal{R}$-boundedness of the resolvent operator}
In this subsection we are going to prove \cref{th:31}. In view of \cref{prop:ev1} and \cref{prop:ev2}, it is enough to prove the following theorem
\begin{theorem} \label{thm:rbd-afs}
Let $1 < q < \infty.$ There exist $\mu_{0} > 0$  and $\theta \in (\pi/2, \pi)$ such that $\mu_{0} + \Sigma_{\theta} \subset \rho(\mathcal{A}_{FS})$ and 
\begin{align}
\mr_{\mathcal{L}(\mx)} \left\{ \lambda (\lambda I - \mathcal{A}_{FS})^{-1} \mid \lambda \in \mu_{0} + \Sigma_{\theta} \right\}  \leqslant C.
\end{align}
\end{theorem}

\begin{proof}
We write $\mathcal{A}_{FS} $   in the form $\mathcal{A}_{FS} = \widetilde {\mathcal{A}}_{FS} + B_{FS}$ where
\begin{align*}
\widetilde{\mathcal{A}}_{FS}  = \begin{pmatrix}  A_{0}  &    -  A_{0}\mpp D \\  0 & 0  \end{pmatrix}, \quad
B_{FS} = \begin{pmatrix}  0  &    0 \\  \mathbb{K}^{-1} \mathcal{C}_{1} & \mathbb{K}^{-1} \mathcal{C}_{2}  \end{pmatrix}.
\end{align*}
We first show that $\widetilde{\mathcal{A}}_{FS}$ with $\mathcal{D}(\widetilde{\mathcal{A}}_{FS}) = \mathcal{D}(\mathcal{A}_{FS})$ is a $\mr$-sectorial operator on $\mx$. Observe that
\begin{align*}
\lambda (\lambda I  - \widetilde{\mathcal{A}}_{FS})^{-1} =  \begin{pmatrix}
\lambda (\lambda I -  A_{0})^{-1} & -(\lambda I -  A_{0})^{-1} A_{0} \mpp D \\ 0 & I
\end{pmatrix}.
\end{align*}
Since
$$
-(\lambda I -  A_{0})^{-1} A_{0} \mpp D = -\lambda(\lambda I - A_{0})^{-1} \mpp D + \mpp D,
$$
we get
\begin{align*}
\lambda (\lambda I  - \widetilde{\mathcal{A}}_{FS})^{-1} =  \begin{pmatrix}
\lambda (\lambda I -  A_{0})^{-1} &-\lambda(\lambda I - A_{0})^{-1}\mpp  D + \mpp D \\ 0 & I
\end{pmatrix}.
\end{align*}
Therefore by \cref{{prop:rsecS}} and \cref{rem:Rbd}, there exists $\theta \in (\pi/2,\pi)$ such that
\begin{align} \label{eq:rb1}
\mathcal{R}_{\mathcal{L}(\mx)} \left\{ \lambda(\lambda I  - \widetilde{{\mathcal A}}_{FS})^{-1} \mid \lambda \in \Sigma_{\theta} \right\} \leqslant C.
\end{align}

Let us now show that,  $\mathcal{C}_{1} \in \mathcal{L}(\mathcal{D}(\mathcal{A}_{FS}), \ct\times \ct).$ By \cref{lem:eqiv-norm}, for any $(\mpp u, \ell, \omega) \in \mathcal{D}(\mathcal{A}_{FS})$ we have $(\mpp u, \ell, \omega) \in W^{2,q}(\ofo) \times \ct \times \ct.$ Therefore, by trace theorem  $\varepsilon(\mpp u) n \in W^{1-1/q,q}(\poso)$ and hence $\ds \int_{\poso} \varepsilon(\mpp u)n \ \dg \in \ct.$ On the other hand, $\Delta \mpp u \in L^{q}(\ofo)$ and $\mathrm{div} \ \Delta\mpp u = 0.$ Therefore by \cref{Normaltrace}, the term $\Delta \mpp u \cdot n$ belongs to $W^{-1/q,q}(\partial\ofo)$ and satisfies the following condition
\begin{align*}
\left\langle \Delta \mpp u\cdot n , 1 \right\rangle_{W^{-1/q,q}, W^{1-1/q',q'}}  = 0.
\end{align*}
Thus by \cite[Theorem 9.2]{Fab98},  $N (\Delta \mpp u \cdot n) \in W^{1,q}(\ofo)$ and $\ds \int_{\partial\oso} N (\Delta \mpp u \cdot n) n \ \dg \in \ct.$ Other terms of the operator $\mathcal{C}_{1}$ can be checked in a similar manner.  Thus $\mathcal{C}_{1} \in \mathcal{L}(\mathcal{D}(\mathcal{A}_{FS}), \ct\times \ct).$ Similarly, one can easily verify that $\mathcal{C}_{2} \in \mathcal{L}(\ct\times \ct, \ct\times \ct).$ Therefore the operator $B_{FS}$ with $\mathcal{D}(B_{FS}) = \mathcal{D}(\mathcal{A}_{FS})$ is a finite rank operator. By \cite[Chapter III, Lemma 2.16]{Eng-Nag}, $B_{FS}$ is a $\widetilde{\mathcal{A}}_{FS}$-bounded operator with relative bound zero. Finally using \cref{pr:perturb} we conclude the proof of the theorem.
\end{proof}

\section{Exponential stability of linear fluid-structure interaction operator} \label{sec:exp}
The aim of this section is to show that the operator $\mathbb{A}$ or equivalently the operator $\mathcal{A}_{FS}$ generates an exponentially stable semigroup.  More precisely, we prove:

\begin{theorem} \label{th:exp-st}
Let $1 < q < \infty.$ The operator $\mathcal{A}_{FS}$ generates an exponentially stable semigroup $\ds \left( e^{t\mathcal{A}_{FS}}\right)_{t \geqslant 0}$ on $\mx.$ Equivalently,  the operator $\mathbb{A}$ generates an exponentially stable semigroup $\ds \left( e^{t\mathbb{A}}\right)_{t \geqslant 0}$ on $\mathbb{H}^{q}(\Omega)$. In other words, there exist constants $C> 0$ and $\eta_{0} > 0$ such that
\begin{equation}
\left\|e^{t\mathcal{A}_{FS}}(u_{0}, \ell_{0}, \omega_{0})^{\top} \right\|_{\mathcal{X}} \leqslant C e^{-\eta_{0} t} \left\|(u_{0}, \ell_{0}, \omega_{0})^{\top}\right\|_{\mathcal{X}}.
\end{equation}
\end{theorem}

To prove this theorem we first show that the set $\left\{ \lambda \in \mathbb{C} \mid\mathrm{Re} \lambda \geqslant 0 \right\}$, i.e, the entire right half plane is contained in the resolvent set of $\mathcal{A_{FS}}$.

\begin{theorem}  \label{th:dd}
Assume $1 < q < \infty$ and $\lambda \in \mathbb{C},$ with $\mathrm{Re} \lambda \geqslant 0.$ Then for any $(f,g_{1},g_{2}) \in \mx,$
the system \eqref{eq:r1} admits a unique solution satisfying the estimate
 \begin{align} \label{est:00}
 \|u\|_{W^{2,q}(\ofo)^{3}} + \|p\|_{W^{1,q}_{m}(\ofo)} + \|\ell\|_{\ct} + \|\omega\|_{\ct} \leqslant C \|(f,g_{1},g_{2})\|_{\mx}.
 \end{align}
\end{theorem}
\begin{proof}
Let us recall, by \cref{prop:ev2}, the  system \eqref{eq:r1}  is equivalent to
\begin{align} \label{d11}
&(\lambda I - \mathcal{A}_{FS}) \begin{pmatrix}
\mpp u \\ \ell \\ \omega
\end{pmatrix} = \begin{pmatrix}
\mpp f \\ \widetilde g_{1} \\ \widetilde g_{2}
\end{pmatrix}, \notag  \\
&(I - \mpp) u = (I - \mpp) D (\ell,\omega ),  \\
&\pi = N (\nu \Delta \mpp u \cdot n) - \lambda N_{S} ((\ell + \omega \times y) \cdot n), \notag
\end{align}
where $(\widetilde g_{1}, \widetilde g_{2})^{\top}  = \mathbb{K}^{-1}(g_{1},g_{2})^{\top}.$
By \cref{thm:rbd-afs},  there exists $\widetilde \lambda > \mu_{0}$ such that $(\widetilde \lambda I  - \mathcal{A}_{FS})$ is invertible. Consequently,  \eqref{d11} can be written as
\begin{align}
&\begin{pmatrix} \label{d12}
\mpp u \\ \ell \\ \omega
\end{pmatrix}  = \left[ I + (\lambda - \widetilde\lambda) (\widetilde\lambda - \mathcal{A}_{FS} )^{-1}\right]^{-1}  \left(\widetilde\lambda - \mathcal{A}_{FS} \right)^{-1} \begin{pmatrix}
\mpp f \\ \widetilde g_{1} \\ \widetilde g_{2}
\end{pmatrix}, \notag \\
&(I - \mpp) u = (I - \mpp) D (\ell,\omega ),  \\
&\pi = N (\nu \Delta \mpp u \cdot n) - \lambda N_{S} ((\ell + \omega \times y) \cdot n). \notag
\end{align}
Since $\ds \left(\widetilde\lambda - \mathcal{A}_{FS} \right)^{-1}$ is a compact operator, in view of Fredholm alternative theorem, the existence and uniqueness of system \eqref{d12} are equivalent. Therefore, in the sequel we show the uniqueness of the solutions \eqref{eq:r1}. Once we prove the uniqueness, the estimate \eqref{est:00} follows easily from \eqref{d12}. Let $(u, \pi, \ell, \omega) \in W^{2,q}(\ofo)^{3} \times W^{1,q}_{m}(\ofo) \times \ct \times \ct$ satisfies the homogeneous system
\begin{alignat}{2} \label{eq:r3}
&\lambda u - \nu \Delta u + \nabla \pi  = 0, \quad \mathrm{div} \ u = 0,  &  \mbox{ in }   \ofo, \notag \\
& u = 0   & \mbox{ on }  \partial\Omega,\notag  \\
&  u  = \ell  + \omega \times y &   \mbox{ on }    \poso, \notag \\
&   \lambda  m \ell =  -  \int_{\partial \Omega_{S}(0)}  \sigma(u,\pi) n \ d\gamma , \\
&  \lambda J(0) \omega =  - \int_{\partial \Omega_{S}(0)} y \times   \sigma(u,\pi)n \ d\gamma.\notag
\end{alignat}
We first show that $(u,\pi) \in W^{2,2}(\ofo)^{3} \times W_{m}^{1,2}(\ofo).$ If $q\geqslant 2,$ this follows from  H\"older's estimate.  Assume $1 < q <2.$ In that case, we can rewrite \eqref{eq:r3} as follows
\begin{align} \label{d14}
&(\widetilde\lambda I - \mathcal{A}_{FS}) \begin{pmatrix}
\mpp u \\ \ell \\ \omega
\end{pmatrix} = (\widetilde \lambda - \lambda) \begin{pmatrix}
\mpp u \\ \ell \\ \omega
\end{pmatrix}, \notag \\
&(I - \mpp) u = (I - \mpp) D (\ell,\omega ),  \\
&\pi = N (\nu \Delta \mpp u \cdot n) - \lambda N_{S} ((\ell + \omega \times y) \cdot n). \notag
\end{align}
Since $W^{2,q}(\ofo) \subset L^{2}(\ofo)$ and $(\widetilde\lambda I - \mathcal{A}_{FS}) $ is invertible, we deduce that $(u,\pi) \in W^{2,2}(\ofo)^{3} \times W_{m}^{1,2}(\ofo).$

Multiplying $\eqref{eq:r3}_{1}$ by $\overline{u},$ $\eqref{eq:r3}_{4}$ by $\overline \ell$ and $\eqref{eq:r3}_{5}$ by $\overline \omega,$ we obtain after integration by parts,
\begin{align*}
\lambda \int_{\ofo} |u|^{2} \ {\rm d} y+ 2\nu \int_{\ofo} \varepsilon(u) : \varepsilon(\overline u) \ {\rm d} y + \lambda m |\ell|^{2} + \lambda J(0)\omega \cdot \overline \omega = 0.
\end{align*}
Taking real part of the above equation we obtain
\begin{align*}
\mathrm{Re}\lambda \int_{\ofo} |u|^{2} \ {\rm d} y + 2\nu \int_{\ofo} |\varepsilon(u)|^{2} \ {\rm d} y  + \mathrm{Re}\lambda m |\ell|^{2} + \mathrm{Re}(\lambda J(0)\omega \cdot \overline \omega )= 0.
\end{align*}
Since $\mathrm{Re} \lambda \geqslant 0,$ we have
\begin{equation*}
2\nu \int_{\ofo} |\varepsilon(u)|^{2} \ {\rm d} y = 0.
\end{equation*}
The above estimate and the fact that $u = 0$ on $\partial\Omega$ imply that $u = 0.$ Next, using $u = \ell + \omega \times y$ for $y \in \poso,$ we get $\ell = \omega = 0.$ Finally, as  $\pi \in W^{1,q}_{m}(\ofo),$ we have $\pi = 0.$
\end{proof}

\begin{proof}[Proof of \cref{th:exp-st}]
From \cref{th:dd}, we have
\begin{equation*}
\left\{ \lambda \in \mathbb{C} \mid \mathrm{Re} \ \lambda \geqslant 0\right\} \subset \rho(\mathcal{A}_{FS}).
\end{equation*}
Also, by \cref{thm:rbd-afs} we have the existence of a constant $C > 0$ such that for any $\lambda \in \mu_{0} + \Sigma_{\theta}$ with $\theta \in (\pi/2,\pi),$
\begin{equation*}
\left\|\left(\lambda - \mathcal{A}_{FS} \right)^{-1} \right\|_{\mathcal{L}(\mx)} \leqslant C.
\end{equation*}
Since $\ds \left\{\lambda \in \mathbb{C} \mid \mathrm{Re} \ \lambda \geqslant 0  \right\} \setminus [\mu_{0} + \Sigma_{\theta}]$ is a compact set, we deduce the existence of a constant $C > 0$ such that for any $\lambda \in \mathbb{C}$ with $\mathrm{Re} \ \lambda \geqslant 0$
\begin{equation*}
\left\|\left(\lambda - \mathcal{A}_{FS} \right)^{-1} \right\|_{\mathcal{L}(\mx)} \leqslant C.
\end{equation*}
 This yields that
\begin{equation*}
\left\{\lambda \in \mathbb{C} \mid \mathrm{Re} \ \lambda \geqslant - \eta \right\} \subset \rho(\mathcal{A}_{FS}),
\end{equation*}
for some $\eta > 0.$ As $\mathcal{A}_{FS}$ generates an analytic semigroup, applying Proposition 2.9 of  \cite[Part II, Chapter 1, pp 120]{BDDM}, we obtain exponential stability of $\mathcal{A}_{FS}$ in $\mx.$
\end{proof}

\section{Maximal $L^{p}$-$L^{q}$ regularity of the system \eqref{fsi-l} } \label{sec:maxlp}

In this section we prove the maximal $L^{p}$-$L^{q}$ regularity of a version of the system  the \eqref{fsi-l} with non zero divergence. Treating a non zero divergence term will be useful in the next section in order to tackle some terms coming from a simple change of variables. More precisely, we consider the system
\begin{alignat}{2} \label{fsi-l-d}
&\partial_{t} u - \nu \Delta u + \nabla \pi  = f, \quad \mathrm{div} \ u = \mathrm{div}\  h  & \quad t \in   (0,\infty),  \ y \in  \ofo, \notag \\
& u = 0  & t \in  (0,\infty), \ y \in \partial\Omega,\notag  \\
&  u  = \ell  + \omega \times y & \quad t\in  (0,\infty), \  y \in  \poso, \notag \\
&  m\frac{d}{dt} \ell =  -  \int_{\partial \Omega_{S}(0)}  \sigma(u,\pi) n \ d\gamma  +  g_{1} & \quad t \in (0,\infty), \\
&  J(0)\frac{d}{dt} \omega =  - \int_{\partial \Omega_{S}(0)} y \times   \sigma(u,\pi)n \ d\gamma + g_{2} &  \quad t \in (0,\infty), \notag \\
&  u (0,y) = u_{0}(y) & y\in   \Omega_{F}(0),  \notag \\
&  \ell(0) = \ell_{0}, \quad  \omega(0) = \omega_{0}. \notag
\end{alignat}
We set
\begin{align*}
W^{2,1}_{q,p} (Q^{F}_{\infty}) =  L^{p}(0,\infty;W^{2,q}(\ofo)) \cap W^{1,p}(0,\infty;L^{q}(\ofo)),
\end{align*}
with
\begin{align*}
\|u\|_{W^{2,1}_{q,p} (Q^{F}_{\infty})} := \|u\|_{L^{p}(0,\infty;W^{2,q}(\ofo))} + \|u\|_{W^{1,p}(0,\infty;L^{q}(\ofo))}.
\end{align*}
We prove the following theorem
\begin{theorem} \label{lplq-L}
Let $1 < p,q < \infty$ such that $\ds \frac{1}{p} + \frac{1}{2q} \neq 1.$ Let $\eta \in [0,\eta_{0}),$ where $\eta_{0}$ is the constant introduced in \cref{th:exp-st}.
Let us also assume that  $\ell_{0} \in \rt,$ $\omega_{0} \in \rt$ and $u_{0} \in B^{2(1-1/p)}_{q,p}(\ofo)$ satisfying the compatibility conditions
\begin{align} \label{cc}
&\mathrm{div} \  u_{0} = 0 \mbox{ in } \ofo,  \notag \\
& u_{0} = \ell_{0} + \omega_{0} \times y \mbox{ on } \poso, \quad u_{0} = 0 \mbox{ on } \partial\Omega \mbox{ if } \frac{1}{p} + \frac{1}{2q} < 1  \\
\mbox{ and } & u_{0} \cdot n = (\ell_{0} + \omega_{0} \times y) \cdot n \mbox{ on } \poso, \quad \quad u_{0} \cdot n = 0 \mbox{ on } \partial \Omega \mbox{ if } \frac{1}{p} + \frac{1}{2q} > 1. \notag
\end{align}
Then for any $e^{\eta t} f \in L^{p}(0,\infty;L^{q}(\ofo))^{3},$ $e^{\eta t} h \in W^{2,1}_{q,p} (Q^{F}_{\infty})^{3},$ $e^{\eta t} g_{1} \in L^{p}(0,\infty;\rt)$ and $ e^{\eta t} g_{2} \in  L^{p}(0,\infty;\rt)$ satisfying
\[ h(0,y) = 0 \mbox{ for all } (t,y) \in (0,\infty) \times \ofo \mbox{ and } h|_{\partial\ofo} = 0,
\]the system \eqref{fsi-l} admits a unique strong solution
\begin{gather*}
e^{\eta t}u \in L^{p}(0,\infty;W^{2,q}(\ofo)^{3}) \cap W^{1,p}(0,\infty;L^{q}(\ofo)^{3}) \\
 e^{\eta t} \pi \in L^{p}(0,\infty;W^{1,q}_{m}(\ofo)) \\
 e^{\eta t} \ell  \in W^{1,p}(0,\infty;\rt), \quad e^{\eta t} \omega  \in W^{1,p}(0,\infty;\rt).
\end{gather*}
Moreover, there exists a constant $C_{L} > 0$ depending only on $\Omega, p$ and $q$ such that
\begin{multline} \label{est:L}
\|e^{\eta (\cdot)}u \|_{W^{2,1}_{q,p} (Q^{F}_{\infty})^{3}}
+ \|e^{\eta (\cdot)}\pi \|_{L^{p}(0,\infty;W^{1,q}(\ofo))}
+ \|e^{\eta (\cdot)}\ell \|_{L^{p}(0,\infty;\rt)} \\ + \|e^{\eta (\cdot)}\omega \|_{L^{p}(0,\infty;\rt)} \leqslant C_{L} \Big(\|u_{0}\|_{B^{2(1-1/p)}_{q,p}(\ofo)} + \|\ell_{0}\|_{\rt} + \|\omega_{0}\|_{\rt} \\
+ \|e^{\eta (\cdot)} f \|_{L^{p}(0,\infty;L^{q}(\ofo))} + \|e^{\eta (\cdot)} h\|_{W^{2,1}_{q,p} (Q^{F}_{\infty})^{3}} + \|e^{\eta (\cdot)} g_{1} \|_{L^{p}(0,\infty;\rt)} + \|e^{\eta (\cdot)} g_{2} \|_{L^{p}(0,\infty;\rt)}\Big).
\end{multline}
\end{theorem}

\begin{proof}
We first consider the case $\eta = 0.$ Let us set $v = u - h.$ Then $(v,\pi, \ell, \omega)$ satisfies the following system
\begin{alignat}{2} \label{fsi-l-d1}
&\partial_{t} v - \nu \Delta v + \nabla \pi  =  F, \quad \mathrm{div} \ v =   0  & \quad \mbox{ in }  (0,\infty) \times \ofo, \notag \\
& v = 0  & \mbox{ on } (0,\infty) \times \partial\Omega,\notag  \\
&  v  = \ell  + \omega \times y & \quad \mbox{ on }    (0,\infty) \times \poso, \notag \\
&  m\frac{d}{dt} \ell =  -  \int_{\partial \Omega_{S}(0)}  \sigma(v,\pi) n \ d\gamma  +  G_{1} & \quad t \in (0,\infty), \\
&  J(0)\frac{d}{dt} \omega =  - \int_{\partial \Omega_{S}(0)} y \times   \sigma(v,\pi)n \ d\gamma + G_{2} &  \quad t \in (0,\infty), \notag \\
&  v (0,y) = u_{0}(y) & \mbox{ in }  \Omega_{F}(0),  \notag \\
&  \ell(0) = \ell_{0}, \quad  \omega(0) = \omega_{0}, \notag
\end{alignat}
where
\begin{align*}
F = f -  \partial_{t} h + \nu \Delta h, \quad G_{1} = g_{1}  - \int_{\poso} \varepsilon(h)n \ \dg, \quad  G_{2} = g_{2}  - \int_{\poso} y \times \varepsilon(h)n \ \dg.
\end{align*}
Proceeding as \cref{prop:ev2}, it is easy to see that, the above system is equivalent to
\begin{align} \label{eq:55}
\begin{cases}
\ds \frac{d}{dt} \begin{pmatrix}
\mpp v \\ \ell \\ \omega
\end{pmatrix} = \mathcal{A}_{FS}  \begin{pmatrix}
\mpp v \\ \ell \\ \omega
\end{pmatrix} +  \begin{pmatrix}
\mpp F \\ \widetilde G_{1} \\ \widetilde G_{2}
\end{pmatrix}, \quad  \begin{pmatrix}
\mpp v(0) \\ \ell(0) \\ \omega (0)
\end{pmatrix} = \begin{pmatrix}
\mpp u_{0} \\ \ell_{0} \\ \omega_{0}
\end{pmatrix},   \\
 (I - \mpp) v = (I - \mpp) D(\ell, \omega), 
\end{cases}
\end{align}
where
\begin{align*}
\widetilde G_{1} =  \int_{\poso} N((F -  \nabla \varphi)\cdot n) n \ \dg + G_{1},  \quad \widetilde G_{2} = \int_{\poso} y \times N((F -  \nabla \varphi)\cdot n) n \ d\gamma + G_{2},
\end{align*}
and $\varphi$ is the solution of the problem
\begin{align*}
-\Delta \varphi = \mathrm{div} \ F \mbox{ in } \ofo, \quad \varphi = 0 \mbox{ on } \partial \ofo.
\end{align*}
(see also \cite[Section 4.2]{Raymond-fsi}  or \cite[Proposition 3.7]{MR17}). The operator $\mathcal{A}_{FS}$ is defined as in \eqref{op:afs}. Let us recall that the operator $\mathcal{A}_{FS}$ is an $\mr$-sectorial operator in $\mx$ (\cref{thm:rbd-afs}).  One can easily verify that, under the hypothesis of the theorem, $(\mpp F, \widetilde G_{1}, \widetilde G_{2}) \in L^{p}(0,\infty;\mx)$ and
\begin{multline*}
\|(\mpp F, \widetilde G_{1}, \widetilde G_{2})\|_{L^{p}(0,\infty;\mx)} \leqslant C \Big(  \| f \|_{L^{p}(0,\infty;L^{q}(\ofo))} + \| h\|_{W^{2,1}_{q,p} (Q^{F}_{\infty})^{3}} \\
+ \| g_{1} \|_{L^{p}(0,\infty;\rt)} + \| g_{2} \|_{L^{p}(0,\infty;\rt)}\Big).
\end{multline*}
From \cite[Theorem 3.4]{Amann00}, we obtain $(\mpp u_{0}, \ell_{0}, \omega_{0}) \in  (\mx, \mathcal{D}(\mathcal{A}_{FS}))_{1-1/p,p}.$ Then by \cref{th:ab-i},  the system \eqref{eq:55} admits a unique solution
$(\mpp v, \ell, \omega) \in L^{p}(0,\infty;\mathcal{D}(\mathcal{A}_{FS})) \cap W^{1,p}(0,\infty;\mx).$  From the expression of $(I-\mpp)v$ together with \cref{lem:eqiv-norm}, one can easily check that $v \in W^{2,1}_{q,p}(Q_{\infty}^{F})^{3}$ and thus  $u \in W^{2,1}_{q,p}(Q_{\infty}^{F})^{3}.$ The estimate \eqref{est:L} is easy to obtain.

The case $\eta > 0$ can be reduced to the previous case by multiplying all the function by $e^{\eta t}$ and using the fact that $\mathcal{A}_{FS} + \eta$ generates an $C^{0}$-semigroup of negative type for all $\eta \in (0, \eta_{0}).$
\end{proof}

\section{Global in time existence and uniqueness} \label{sec:gr}
In this section we are going to prove \cref{th-g}. As the domain of the fluid equation for the full nonlinear problem is also a unknown of the problem, we first rewrite the system in a fixed spatial domain.
\subsection{Change of variables}
We describe a change of variable to rewrite the system \eqref{eq:mainsys} in a fixed spatial domain. We follow the approach of \cite{Cum08}. Let us assume that \eqref{eq:ini-d} is satisfied and we also assume
\begin{align} \label{eq:c1}
 \|a\|_{L^{\infty}(0,\infty;\rt)} + \|Q - I_{3}\|_{L^{\infty}(0,\infty;\mathbb{R}^{3 \times 3})}  \mathrm{diam} (\oso) \leqslant \frac{\alpha}{2}.
\end{align}
With the above choice  we have $\mathrm{dist} \left(\ost,\partial\Omega\right) \geqslant \alpha/2  $ for all $t \in [0,\infty).$ We consider a cut-off function $\psi$ which satisfies
\begin{align}
\psi \in C^{\infty}(\overline \Omega), \quad \psi = 1 \mbox{ if } \mathrm{dist}(x,\partial \Omega) > \alpha/4, \quad \psi = 0 \mbox{ if } \mathrm{dist}(x,\partial \Omega) < \alpha/8.
\end{align}
We introduce a function $\xi$ defined in $(0,\infty) \times \Omega$ by
\begin{align*}
\xi(t,x) = a'(t) + (x - a(t)) + \frac{|x-a(t)|^{2}}{2} \omega(t)
\end{align*}
and $\Lambda$  in $(0,\infty) \times \Omega$ by
\begin{align*}
\Lambda(t,x) =
\psi(x) \; \left(a'(t) + \omega(t) \times (x -a(t)\right) +
\begin{pmatrix}
\ds \frac{\partial \psi}{\partial x_{2}}(x) \xi_{3}(t,x) - \frac{\partial \psi}{\partial x_{3}}(x) \xi_{2}(t,x) \\
\ds \frac{\partial \psi}{\partial x_{3}}(x) \xi_{1} (t,x)- \frac{\partial \psi}{\partial x_{1}}(x) \xi_{3} (t,x)\\
\ds \frac{\partial \psi}{\partial x_{1}}(x) \xi_{2}(t,x) - \frac{\partial \psi}{\partial x_{2}}(x) \xi_{1}(t,x)
\end{pmatrix}.
\end{align*}
With the above definitions, it is easy to see that $\Lambda$ satisfies the following lemma
\begin{lemma}
Let us assume that $a \in W^{2,p}(0,\infty)$ and $\omega \in W^{1,p}(0,\infty).$ Let $\Lambda$ be defined as above. Then we have
\begin{itemize}
\item $\Lambda(t,x) = 0$ for all $t \in [0,\infty)$ and for all $x$ such that $\mathrm{dist}(x,\partial \Omega) < \alpha/8.$
\item $\mathrm{div} \Lambda(t,x) = 0$ for all $t \in [0,\infty)$ and $x \in \Omega.$
\item $\Lambda(t,x) = a'(t) + \omega(t) \times (x - a(t)) $ for all $t \in [0,\infty)$ and $x \in \ost.$
\item $\Lambda \in C([0,\infty) \times \Omega;\rt).$ Moreover, for all $t\in [0,\infty),$  $\Lambda(t,\cdot)$ is a $C^{\infty}$ function for all $x \in \Omega,$ the function $\Lambda(\cdot,x) \in W^{1,p}(0,\infty;\rt).$ 
\end{itemize}
\end{lemma}

Next we consider the characteristic $X$ associated to the flow $\Lambda,$ that is the solution of the Cauchy problem
\begin{align} \label{def:X}
\partial_{t}  X(t,y) & =  \Lambda(t,X(t,y))  \quad (t >0), \notag \\
X(0,y)               &  = y\in \overline{\Omega}.
\end{align}
We have the following lemma
\begin{lemma} \label{lem:jd}
For all $y \in \Omega,$ the initial value problem $\eqref{def:X}$ admits a unique solution $X(\cdot,y) : [0,\infty) \mapsto \rt,$ which is a $C^{1}$ function in $[0,\infty).$  Furthermore $X$ satisfies the following properties
\begin{itemize}
\item For any $t \in [0,\infty),$ $X(t,\cdot)$ is a  $C^{1}$- diffeomorphism  from $\Omega$ onto $\Omega$ and $\ofo$ onto $\oft.$
\item For all $y \in \Omega$ and $t \in [0,\infty),$ we have
\[ \mathrm{det} \ \nabla X(t,\cdot) = 1.\]
\item  For each $t \geqslant 0,$  we denote by $Y(t,\cdot) = [X(t,\cdot)]^{-1}$ the
inverse of $X(t,\cdot)$.
\end{itemize}
\end{lemma}
\begin{proof}
See \cite[Lemma 2.2]{Cum08}.
\end{proof}

 We consider the following change of
variables
\begin{gather}
 \widetilde  u (t,y)   = Q^{-1}(t)u(t,X(t,y)), \qquad \widetilde  \pi(t,y)  = \pi(t,X(t,y)) \label{lpq00} \\
\widetilde  \ell(t)   = Q^{-1}(t) \dot a(t), \qquad \widetilde  \omega(t)  = Q^{-1}(t) \omega(t), \label{lpq03}
\end{gather}
for $(t,y) \in (0,\infty) \times \ofo$. %and where $\bar \rho > 0 $ and $\bar \theta > 0$.

 Then $(\widetilde u, \widetilde p, \widetilde \ell, \widetilde \omega)$ satisfies the following system
\begin{alignat}{2} \label{eq:NL-1}
&\partial_{t} \widetilde u - \nu \Delta \widetilde u + \nabla \widetilde p  = \mathcal{F}, \quad \mathrm{div} \ u = \mathrm{div} \ \mathcal{H}, &  \quad t \in  (0,\infty), \ y \in \ofo, \notag \\
& \widetilde u = 0,  & t \in (0,\infty),  y \in  \partial\Omega,\notag  \\
&  \widetilde u  = \widetilde \ell  + \widetilde \omega \times y, & \quad t \in   (0,\infty), \ y \in \poso, \notag \\
&  m\widetilde {\ell}' =  -  \int_{\partial \Omega_{S}(0)}  \sigma(\widetilde u,\widetilde p) n \ d\gamma  +  \mathcal{G}_{1}, & \quad t \in (0,\infty), \\
&  J(0) \widetilde\omega '=  - \int_{\partial \Omega_{S}(0)} y \times   \sigma(\widetilde u,\widetilde p)n \ d\gamma + \mathcal{G}_{2}, &  \quad t \in (0,\infty), \notag \\
&  u (0,y) = u_{0}(y) & y \in  \Omega_{F}(0),  \notag \\
&  \ell(0) = \ell_{0},  \quad  \omega(0) = \omega_{0},  \notag
\end{alignat}
where
\begin{align} \label{def-Q-g}
\dot Q  =  Q A(\tilde \omega), \quad Q(0) = I_{3},
\end{align}
is the rotation matrix of the solid at instant $t$,
\begin{align} \label{Jacobi-g}
X(t,y)  = y + \int_{0}^{t} \Lambda(s,X(s,y))  \ {\rm d}s, \quad \mbox{ and } \quad \nabla Y(t,X(t,y)) = [\nabla X]^{-1}(t,y),
\end{align}
for every $y\in \Omega_F(0)$ and $t\geqslant 0$. Using  the notation
\begin{align} \label{Z-g}
Z(t,y) = \left( Z_{i,j}\right)_{1\leqslant i, j\leqslant 3}=[\nabla X]^{-1}(t,y) \qquad\qquad(t\geqslant 0,\ y\in \Omega_F(0)),
\end{align}
the remaining terms in \eqref{eq:NL-1} are defined by
\begin{align} \label{eq:F}
\mathcal{F}_{i}(\widetilde u, \widetilde \pi, \widetilde \ell, \widetilde \omega) & = - [(Q-I_{3}) \partial_{t} \widetilde u]_{i} - (\omega \times Q\widetilde u)_{i}  + \partial_{t} X  \cdot Z^{T} \nabla (Q\widetilde u)_{i} - \widetilde u  \cdot Z^{T} \nabla (Q\widetilde u)_{i} \\
& + \nu\sum_{l,j,k} \frac{\partial^2 (Q\tilde u)_{i} }{\partial y_l\partial y_k}  \left(Z_{k,j}  - \delta_{k,j}\right) Z_{l,j}  + \nu\sum_{l,k} \frac{\partial^2 (Q\tilde u)_{i} }{\partial y_l\partial y_k} \left( Z_{l,k} - \delta_{l,k} \right)
\notag \\
& + \nu \left[ (Q - I) \Delta \tilde u\right]_{i} + \nu \sum_{l,j,k} Z_{l,j} \frac{\partial  (Q \tilde u)_i}{\partial y_k} \frac{\partial Z_{k,j}}{\partial y_{l}} - \left( (Z^{\top} - I_{3}) \nabla\widetilde p\right)_{i}, \notag
\end{align}
\begin{align} \label{eq:H}
\mathcal{H}(\widetilde u, \widetilde \pi, \widetilde \ell, \widetilde \omega) = (I_{3} - [ZQ]^{T}) \widetilde u,
\end{align}
\begin{align} \label{eq:G}
\mathcal{G}_{1}( \widetilde \ell, \widetilde \omega) = - m (\widetilde \omega \times \widetilde \ell), \qquad \qquad \mathcal{G}_{2}( \widetilde \ell, \widetilde \omega) = J(0) \widetilde \omega \times \widetilde \omega.
\end{align}

\subsection{Estimate of nonlinear terms}
In this section, we are going to estimate the nonlinear terms $\mathcal{F},\mathcal{H},\mathcal{G}_{1}$ and $\mathcal{G}_{2}$ defined as in \eqref{eq:F} - \eqref{eq:G}.

\medskip

Throughout this section we assume $1 < p,q < \infty$ satisfying the conditions $\ds \frac{1}{p} + \frac{1}{2q} \neq 1$ and $\ds \frac{1}{p} + \frac{3}{2q} \leqslant \frac{3}{2}.$ Let $p'$
denote the conjugate of $p$, i.e.,
$\ds\frac{1}{p} + \frac{1}{p'} = 1$.  Let us fix  $\eta \in (0, \eta_{0})$, where $\eta_{0}$ is the constant introduced in  \cref{th:exp-st} and  we introduce the following ball
\begin{align*}
{\mathcal S}_{\gamma} = \Big\{(\widetilde u,\widetilde \pi,\widetilde \ell,\widetilde \omega) \mid \widetilde \rho(t,y) =
\left\| (\widetilde u, \widetilde \pi,\widetilde \ell,\widetilde \omega)\right\|_{\mathcal{S}} \leqslant \gamma
 \Big\},
\end{align*}
where
\begin{multline} \label{ball-g}
\left\| (\widetilde u,\widetilde p,\widetilde \ell,\widetilde \omega)\right\|_{\mathcal{S}}:= \|e^{\eta (\cdot)}\widetilde u \|_{L^{p}(0,\infty;W^{2,q}(\ofo))^{3}} + \|e^{\eta (\cdot)}\widetilde u \|_{W^{1,p}(0,\infty;L^{q}(\ofo))^{3}} \\ +
\|e^{\eta (\cdot)}\widetilde p \|_{L^{p}(0,\infty;W^{1,q}(\ofo))}
+ \|e^{\eta (\cdot)}\widetilde \ell \|_{W^{1,p}(0,\infty;\rt)} + \|e^{\eta (\cdot)}\widetilde\omega \|_{W^{1,p}(0,\infty;\rt)}.
\end{multline}

Our aim is to estimate the nonlinear terms  in  \eqref{eq:F} - \eqref{eq:G}.
\begin{proposition} \label{prop-nle}
Let us assume $1 < p,q < \infty$ satisfying the condition  $\ds \frac{1}{p} + \frac{3}{2q} \leqslant \frac{3}{2}.$ There exist constants  $\gamma_{0} \in (0,1)$ and $C_{N} > 0$ both depending only on $p,q,\eta$ and $\Omega_{F}(0)$ such that for every $\gamma \in (0, \gamma_{0})$ and for every  $(\widetilde u,\widetilde \pi,\widetilde \ell,\widetilde \omega)  \in {\mathcal S}_{\gamma}$, we have
\begin{multline} \label{e-nl}
\|e^{\eta (\cdot)}\mathcal{F}\|_{L^{p}(0,\infty;L^{q}(\ofo))} + \|e^{\eta (\cdot)}\mathcal{H}\|_{W^{2,1}_{q,p}(Q_{F}^{\infty})} \\ + \|e^{\eta (\cdot)}\mathcal {G}_{1}\|_{L^{p}(0,\infty;\rt)} + \|e^{\eta (\cdot)}\mathcal {G}_{2}\|_{L^{p}(0,\infty;\rt)} \leqslant  C_{N} \gamma^{2}.
\end{multline}
\end{proposition}

\begin{proof}
The constants appearing in this proof will be denoted by $C$ and depends only on $p,q,\eta$ and $\Omega_{F}(0).$  Let us first show that,  there exists $\gamma_{0} \in (0,1),$ such that, for every $\gamma \in (0, \gamma_{0})$ and for every  $(\widetilde u,\widetilde \pi,\widetilde \ell,\widetilde \omega)  \in {\mathcal S}_{\gamma}$ the condition \eqref{eq:c1} is verified.

The solution of \eqref{def-Q-g} satisfies $Q \in SO(3)$ and thus $|Q(t)| = 1$ for all $t\geqslant 0.$
We can rewrite $Q$ as follows
\begin{equation*}
Q(t) = I + \int_{0}^{t} e^{-\eta s} e^{\eta s} \widetilde \omega(s) \times Q(s)  \  ds.
\end{equation*}
Therefore
\begin{align} \label{q1}
\|Q - I_{3}\|_{L^{\infty}(0,\infty;\mathbb{R}^{3\times 3})} &\leqslant  \int_{0}^{\infty} e^{-\eta s} e^{\eta s} |\widetilde \omega(s)| \ ds \notag \\
& \leqslant \left(\int_{0}^{\infty} e^{-p'\eta t } \ dt \right)^{1/p'} \|e^{\eta (\cdot)}\widetilde\omega \|_{L^{p}(0,\infty;\rt)} \leqslant \left( \frac{1}{p'\eta}\right)^{1/p'} \gamma.
\end{align}
Similarly,
\begin{align} \label{a1}
\|a\|_{L^{\infty}(0,\infty;\rt)} \leqslant \int_{0}^{\infty}  e^{-\eta t} e^{\eta t} |Q(s)| |\widetilde \ell(t)| \ dt \leqslant \left( \frac{1}{p'\eta}\right)^{1/p'} \gamma.
\end{align}
Combining \eqref{q1} and \eqref{a1}, we get 
\begin{equation*}
\|a\|_{L^{\infty}(0,\infty;\rt)} + \|Q - I_{3}\|_{L^{\infty}(0,\infty;\mathbb{R}^{3 \times 3})}  \mathrm{diam} (\oso)  \leqslant \gamma \left( \frac{1}{p'\eta}\right)^{1/p'} (1 + \mathrm{diam} (\oso)).
\end{equation*}
Let us set
\begin{align} \label{gam0}
\gamma_{0} = \min\left\{1,   \frac{\alpha}{2C_{p,\eta} (1+ \mathrm{diam}(\ofo))}\right\}, \quad \mbox{ with } C_{p,\eta} = \left( \frac{1}{p'\eta}\right)^{1/p'}.
\end{align}
 With the above choice of $\gamma_{0},$ we can easily verify  the condition \eqref{eq:c1}.

Let $X$ be defined as in \eqref{Jacobi-g}. Differentiating \eqref{Jacobi-g} with respect to $y$ we obtain
\begin{align*}
\nabla X (t,y) = I_{3} +  \int_{0}^{t} \nabla \Lambda(s,X(s,y)) \nabla X(s,y) \ ds
\end{align*}

From the definition of $\Lambda$ and $X$ we obtain
\begin{align*}
&\|\nabla X(t,\cdot)\|_{C^{2}(\Omega)} \\
&\leqslant 1 + C \int_{0}^{t} e^{-\eta s} e^{\eta s} \left(|\widetilde \omega(s)+ |\widetilde \ell(s)|\right)  \|\nabla X(s,\cdot)\|_{C^{\infty}(\Omega)} \ ds \\
&\leqslant 1 + C \left( \|e^{\eta (\cdot)}\widetilde \ell \|_{L^{\infty}(0,\infty;\rt)} + \|e^{\eta (\cdot)}\widetilde\omega \|_{L^{\infty}(0,\infty;\rt)}\right) \int_{0}^{t} e^{-\eta s}  \|\nabla X(s,\cdot)\|_{C^{2}(\Omega)}  \ ds \\
& \leqslant 1 + C \int_{0}^{t}e^{-\eta s}  \|\nabla X(s,\cdot)\|_{C^{2}(\Omega)} \ ds,
\end{align*}
 By Gronwall's inequality
\begin{align*}
\|\nabla X(t,\cdot)\|_{C^{2}(\Omega)} \leqslant \mathrm{exp} \left(C  \int_{0}^{t} e^{-\eta s} \ ds  \right)  \leqslant e^{C/\eta} \mbox{ for all } t \in (0,\infty).
\end{align*}
With the above estimate  we obtain
\begin{align} \label{j3}
\|\nabla X(t,\cdot) - I_{3}\|_{L^{\infty}(0,\infty;C^{2}(\Omega))} \leqslant C \int_{0}^{\infty} e^{-\eta s} e^{\eta s} \left(\|\widetilde \omega(s)\|_{\rt}+ \|\widetilde \ell(s)\|_{\rt}\right)  \ ds \leqslant C \gamma.
\end{align}
It is also easy to see that
\begin{align*}
\|\mathrm{Cof} \nabla X \|_{L^{\infty}(0,\infty;C^{2}(\Omega))} \leqslant C.
\end{align*}
From \cref{lem:jd}, we have $\mathrm{det} \nabla X(t,y) = 1$ for all $t \geqslant 0$ and $y \in \overline{\Omega}.$ 
Thus from the relation
\[ \ds Z = [\nabla X]^{-1} =  \frac{1}{\mathrm{det} \nabla X} \mathrm{Cof} \nabla X, \]
we obtain
\begin{align} \label{z12}
\|Z\|_{L^{\infty}(0,\infty;C^{2}(\Omega))} \leqslant C.
\end{align}
Using the above estimate and \eqref{j3}, we get
\begin{align} \label{z2}
\|Z - I_{3}\|_{L^{\infty}(0,\infty;C^{2}(\Omega))} \leqslant \|Z\|_{L^{\infty}(0,\infty;C^{2}(\Omega))}  \|\nabla X - I_{3}\|_{L^{\infty}(0,\infty;C^{2}(\Omega))}  \leqslant C\gamma.
\end{align}
In a similar manner we can obtain the following estimates
\begin{align} \label{z3}
&\|\partial_{t} X\|_{L^{\infty}((0,\infty)\times \ofo)}\leqslant C\gamma, \quad  \|\partial_{t} Z\|_{L^{\infty}((0,\infty)\times \ofo)}\leqslant C \gamma, \notag \\
&  \|ZQ - I_{3}\|_{L^{\infty}(0,\infty;C^{2}(\Omega))} \leqslant C\gamma.
\end{align}

We are now in a position to estimate the nonlinear terms.

\underline{Estimate of $\mathcal{F}.$}
\begin{align} \label{f1-e}
\|e^{\eta (\cdot)}\mathcal{F}\|_{L^{p}(0,\infty;L^{q}(\ofo))} \leqslant C \gamma^{2}.
\end{align}

 \noindent \textbullet Estimate of first, second and third term of $\mathcal{F}$:  Using  \eqref{q1}, \eqref{z12} and \eqref{z3} we have
 \begin{align*}
 &\left\| e^{\eta(\cdot)}\Big(- [(Q-I_{3}) \partial_{t} \widetilde u]_{i} -(\omega \times Q\widetilde u)_{i} + \partial_{t} X  \cdot Z^{T} \nabla (Q\widetilde u)_{i}\Big)\right\|_{L^{p}(0,\infty;L^{q}(\ofo))} \\
 &\leqslant C \left( \|Q - I_{3}\|_{L^{\infty}(0,\infty;\mathbb{R}^{3\times 3})} + \|\widetilde \omega\|_{L^{\infty}(0,\infty;\rt)} + \|\partial_{t} X\|_{L^{\infty}((0,\infty)\times \ofo)}\right) \|e^{\eta(\cdot)} \widetilde u\|_{W^{2,1}_{q,p}(Q^{\infty}_{F})} \\
 & \leqslant C\gamma^{2}.
 \end{align*}

 \noindent  \textbullet Estimate of fourth term of $\mathcal{F}$: By H\"older's inequality and using \eqref{z12},  we obtain
 \begin{align*}
 &\|e^{\eta(\cdot)}\widetilde u  \cdot Z^{T} \nabla (Q\widetilde u)_{i}\|_{L^{p}(0,\infty;L^{q}(\ofo))} \notag \\
 & \leqslant  C \|e^{\eta(\cdot)}\widetilde u  \cdot  \nabla \widetilde u_{i}\|_{L^{p}(0,\infty;L^{q}(\ofo))} \notag \\
 & \leqslant C \|e^{\eta(\cdot)}\widetilde u\|_{L^{3p}(0,\infty;L^{3q}(\ofo))} \|\nabla \widetilde u_{i}\|_{L^{3p/2}(0,\infty;L^{3q/2}(\ofo))}.
 \end{align*}
Since $\ds \frac{1}{p} + \frac{3}{2q} \leqslant \frac{3}{2},$ one has the following  embeddings (see for example \cite[Proposition 4.3]{Gei13} )
$$W^{2,1}_{q,p}(Q^{\infty}_{F}) \hookrightarrow L^{3p}(0,\infty;L^{3q}(\ofo)) \mbox{ and } W^{2,1}_{q,p}(Q^{\infty}_{F}) \hookrightarrow L^{3p/2}(0,\infty;W^{1+3q/2}(\ofo)).$$
Therefore, using the above embeddings we obtain
\begin{align*}
\|e^{\eta(\cdot)}\widetilde u  \cdot Z^{T} \nabla (Q\widetilde u)_{i}\|_{L^{p}(0,\infty;L^{q}(\ofo))} \leqslant C\gamma^{2}
\end{align*}
\noindent  \textbullet Estimate of fifth term of $\mathcal{F}$(estimates of remaining terms of  $\mathcal{F}$ are similar) : Using \eqref{z12} and \eqref{z2} we have
\begin{align*}
&\left\| \nu e^{\eta(\cdot)}\sum_{l,j,k} \frac{\partial^2 (Q\tilde u)_{i} }{\partial y_l\partial y_k}  \left(Z_{k,j}  - \delta_{k,j}\right) Z_{l,j} \right\|_{L^{p}(0,\infty;L^{q}(\ofo))} \\
& \leqslant C  \|Z - I_{3}\|_{L^{\infty}(0,\infty;C^{2}(\Omega))} \|e^{\eta(\cdot)} \widetilde u\|_{W^{2,1}_{q,p}(Q^{\infty}_{F})} \leqslant C \gamma^{2}.
\end{align*}

\underline{Estimate of $\mathcal{H}.$}
\begin{align} \label{h-e}
\|e^{\eta (\cdot)}\mathcal{H}\|_{W^{2,1}_{q,p}(Q_{F}^{\infty})} \leqslant C\gamma^{2}.
\end{align}
Using \eqref{z2} and \eqref{z3}, we obtain
\begin{align*}
&\|e^{\eta (\cdot)}(I_{3} - [ZQ]^{T}) \widetilde u\|_{W^{2,1}_{q,p}(Q_{F}^{\infty})} \\
& \leqslant C \Big( \|ZQ - I_{3}\|_{L^{\infty}(0,\infty;C^{2}(\Omega))} + \|\partial_{t} Z\|_{L^{\infty}((0,\infty)\times \ofo)} \\
& \qquad \qquad \qquad + \|\partial_{t} Q\|_{L^{\infty}(0,\infty;\rt \times \rt)} \Big) \|e^{\eta(\cdot)} \widetilde u\|_{W^{2,1}_{q,p}(Q^{\infty}_{F})} \\
& \leqslant C \gamma^{2}.
\end{align*}
\underline{Estimate of $\mathcal{G}_{1}$ and $\mathcal{G}_{2}:$} From the expressions of $\mathcal{G}_{1}$ and $\mathcal{G}_{2}$ it is easy to see that
\begin{align} \label{g-e}
\|e^{\eta (\cdot)}\mathcal {G}_{1}\|_{L^{p}(0,\infty;\rt)} + \|e^{\eta (\cdot)}\mathcal {G}_{2}\|_{L^{p}(0,\infty;\rt)} \leqslant  C \gamma^{2}.
\end{align}
Combining \eqref{f1-e} - \eqref{g-e}, we obtain \eqref{e-nl}.
\end{proof}

\begin{proposition} \label{prop-nlpe}
Let us assume $1 < p,q < \infty$ satisfying the condition  $\ds \frac{1}{p} + \frac{3}{2q} \leqslant \frac{3}{2}.$ Let $\gamma_{0}$ is defined as in \eqref{gam0}.
There exist constant   $C_{lip} > 0$  depending only on $p,q,\eta$ and $\Omega_{F}(0)$ such that for every $\gamma \in (0, \gamma_{0})$ and for every  $(\widetilde u^{j},\widetilde \pi^{j},\widetilde \ell^{j},\widetilde \omega^{j})  \in {\mathcal S}_{\gamma}$,  $j=1,2$ we have
\begin{multline} \label{e-nlp}
\|e^{\eta (\cdot)}\mathcal{F}(\widetilde u^{1},\widetilde \pi^{1},\widetilde \ell^{1},\widetilde \omega^{1}) - e^{\eta (\cdot)}\mathcal{F}(\widetilde u^{2},\widetilde \pi^{2},\widetilde \ell^{2},\widetilde \omega^{2})\|_{L^{p}(0,\infty;L^{q}(\ofo))} \\
+ \|e^{\eta (\cdot)}\mathcal{H}(\widetilde u^{1},\widetilde \pi^{1},\widetilde \ell^{1},\widetilde \omega^{1}) - e^{\eta (\cdot)}\mathcal{H}(\widetilde u^{2},\widetilde \pi^{2},\widetilde \ell^{2},\widetilde \omega^{2})\|_{W^{2,1}_{q,p}(Q_{F}^{\infty})} \\
 + \|e^{\eta (\cdot)}\mathcal {G}_{1}(\widetilde \ell^{1},\widetilde \omega^{1}) - e^{\eta (\cdot)}\mathcal {G}_{1}(\widetilde \ell^{2},\widetilde \omega^{2})\|_{L^{p}(0,\infty;\rt)} + \|e^{\eta (\cdot)}\mathcal {G}_{2}(\widetilde \ell^{1},\widetilde \omega^{1}) - e^{\eta (\cdot)}\mathcal {G}_{2}(\widetilde \ell^{2},\widetilde \omega^{2})\|_{L^{p}(0,\infty;\rt)} \\
 \leqslant  C_{lip} \gamma  \left\|(\widetilde u^{1},\widetilde \pi^{1},\widetilde \ell^{1},\widetilde \omega^{1}) -  (\widetilde u^{2},\widetilde \pi^{2},\widetilde \ell^{2},\widetilde \omega^{2}) \right\|_{\mathcal{S}}
\end{multline}
\end{proposition}
\begin{proof}
The proof is similar to the proof of \cref{prop-nle}.
\end{proof}

\subsection{Proof of  \cref{th-g}}
At first we prove global existence and uniqueness theorem for the transformed system \eqref{eq:NL-1} -\eqref{eq:G} under the smallness assumption on the initial data.  More precisely we prove the following theorem
\begin{theorem} \label{th-g-t}
Let $1 < p,q < \infty$ satisfying the conditions $\ds \frac{1}{p} + \frac{1}{2q} \neq 1$ and $\ds \frac{1}{p} + \frac{3}{2q} \leqslant \frac{3}{2}.$ Let $\eta \in (0,\eta_{0}),$ where $\eta_{0}$ is the constant introduced in \cref{th:exp-st}.
There exist a constant $\widetilde \gamma > 0$ depending only on $p,q, \eta$ and $\ofo$ such that,   for all  $\gamma \in (0,\widetilde \gamma)$ and for all   $(u_{0},\ell_{0},\omega_{0}) \in  B^{2(1-1/p)}_{q,p}(\ofo) \times \rt \times \rt$ satisfying the compatibility conditions
\begin{align*}
&\mathrm{div} \  u_{0} = 0 \mbox{ in } \ofo,  \\
& u_{0} = \ell_{0} + \omega_{0} \times y \mbox{ on } \poso, \quad u_{0} = 0 \mbox{ on } \partial\Omega \mbox{ if } \frac{1}{p} + \frac{1}{2q} < 1  \\
\mbox{ and } & u_{0} \cdot n = (\ell_{0} + \omega_{0} \times y) \cdot n \mbox{ on } \poso, \quad \quad u_{0} \cdot n = 0 \mbox{ on } \partial \Omega \mbox{ if } \frac{1}{p} + \frac{1}{2q} > 1,
\end{align*}
and
\begin{align} \label{eq:ini-ball}
\|u_{0}\|_{B^{2(1-1/p)}_{q,p}(\ofo)} + \|\ell_{0}\|_{\rt} + \|\omega_{0}\|_{\rt} \leqslant \frac{\gamma}{2C_{L}},
\end{align}
where $C_{L}$ is the continuity constant appear in \eqref{est:L},
the system \eqref{eq:NL-1} -\eqref{eq:G} admits a unique strong solution $(\widetilde u, \widetilde \pi,\widetilde\ell, \widetilde\omega)$ such that
\begin{align}
\left\|(\widetilde u,\widetilde \pi,\widetilde \ell,\widetilde \omega)  \right\|_{\mathcal{S}} \leqslant \gamma.
\end{align}
\end{theorem}

\begin{proof}
Let us set
\begin{align}
\widetilde \gamma = \min \left\{ \gamma_{0}, \frac{1}{2C_{L} C_{N}}, \frac{1}{2C_{L} C_{lip}}\right\}
\end{align}
where $\gamma_{0}$ is defined as in \eqref{gam0} and $C_{L},$ $C_{N}$ and $C_{lip}$ are the constants  appearing \cref{lplq-L}, \cref{prop-nle} and \cref{prop-nlpe} respectively. Let us choose $\gamma \in (0,\widetilde \gamma)$ and  $(v, \varphi, \kappa, \tau) \in \mathcal{S}_{\gamma}.$ We consider the following problem
\begin{alignat}{2} \label{eq:NL-2}
&\partial_{t} \widetilde u - \nu \Delta \widetilde u + \nabla \widetilde \pi  = \mathcal{F}(v, \varphi, \kappa, \tau) , \quad \mathrm{div} \ u = \mathrm{div} \ \mathcal{H}(v, \varphi, \kappa, \tau),  & \quad t\in   (0,\infty), \ y \in\ofo, \notag \\
& \widetilde u = 0,  & t\in   (0,\infty), \ y \in \partial\Omega,\notag  \\
&  \widetilde u  = \widetilde \ell  + \widetilde \omega \times y, & \quad t\in  (0,\infty),  \ y \in \poso, \notag \\
&  m\widetilde \ell' =  -  \int_{\partial \Omega_{S}(0)}  \sigma(\widetilde u,\widetilde \pi) n \ d\gamma  +  \mathcal{G}_{1}(\kappa, \tau), & \quad t \in (0,\infty), \\
&  J(0) \widetilde\omega '=  - \int_{\partial \Omega_{S}(0)} y \times   \sigma(\widetilde u,\widetilde \pi)n \ d\gamma + \mathcal{G}_{2}(\kappa, \tau), &  \quad t \in (0,\infty), \notag \\
&  u (0,y) = u_{0}(y), & y \in  \Omega_{F}(0),  \notag \\
&  \ell(0) = \ell_{0},  \quad  \omega(0) = \omega_{0}.  \notag
\end{alignat}
We are going to show the mapping
\begin{align*}
\mathcal{N}: (v, \varphi, \kappa, \tau)  \mapsto (\widetilde u,\widetilde \pi,\widetilde \ell,\widetilde \omega)
\end{align*}
where $(\widetilde u,\widetilde \pi,\widetilde \ell,\widetilde \omega) $ is the solution to the system \eqref{eq:NL-2}, is a contraction in $\mathcal{S}_{\gamma}.$  As $(v, \varphi, \kappa, \tau) \in \mathcal{S}_{\gamma},$ we can apply \cref{lplq-L} and \cref{prop-nle} to the system \eqref{eq:NL-2} and using \eqref{eq:ini-ball}  and definition of $\widetilde \gamma$ we obtain
\begin{align*}
&\left\|\mathcal{N}(v, \varphi, \kappa, \tau)  \right\|_{\mathcal{S}}\\
& \leqslant C_{L} \Big(\|u_{0}\|_{B^{2(1-1/p)}_{q,p}(\ofo)} + \|\ell_{0}\|_{\rt} + \|\omega_{0}\|_{\rt} \Big) + C_{L} C_{N} \gamma^{2} \\
& \leqslant  \gamma.
\end{align*}
Thus $\mathcal{N}$ is a mapping from $\mathcal{S}_{\gamma}$ to itself for all $\gamma \in (0,\widetilde \gamma).$ Next, using \cref{lplq-L} and \cref{prop-nlpe}, we obtain
\begin{align*}
&\left\|\mathcal{N}(v^{1}, \varphi^{1}, \kappa^{1}, \tau^{1})  -  \mathcal{N}(v^{2}, \varphi^{2}, \kappa^{2}, \tau^{2}) \right\|_{\mathcal{S}} \\
& \leqslant C_{L} C_{N} \gamma \left\|(v^{1}, \varphi^{1}, \kappa^{1}, \tau^{1})  -  (v^{2}, \varphi^{2}, \kappa^{2}, \tau^{2}) \right\|_{\mathcal{S}},
\end{align*}
for all $(v^{j}, \varphi^{j}, \kappa^{j}, \tau^{j}),$ $j=1,2.$  Again using the definition of $\widetilde \gamma$ one can easily verify that $\mathcal{N}$ is a strict contraction of $\mathcal{S}_{\gamma}$ for any $\gamma \in (0, \widetilde \gamma),$ which implies our existence and uniqueness result.
\end{proof}

\textit{Proof of \cref{th-g}:}
Let $(\widetilde u, \widetilde \pi,\widetilde\ell, \widetilde\omega)$ be the solution of the system \eqref{eq:NL-1} -\eqref{eq:G}, constructed in \cref{th-g-t}. Since $\gamma < \widetilde \gamma,$ \eqref{eq:c1} is verified and $X(t,\cdot)$ is a well defined mapping and it is a $C^{1}$-diffeomorphism from $\Omega_{F}(0)$ into $\oft.$ Therefore, there is a unique $Y(t,\cdot)$ from $\Omega_{F}(t)$  into $\ofo$ such that $Y(t,\cdot)= X(t,\cdot)^{-1}.$ We set, for all $t \geqslant 0$ and $x \in \oft$
\begin{gather*}
u(t,x) = \widetilde u (t, Y(t,x)), \quad \pi(t,x) = \widetilde \pi(t, Y(t,x)), \\
 a'(t) = Q(t) \widetilde \ell(t) \mbox{ and } \omega(t)  = Q(t) \widetilde \omega(t).
\end{gather*}
We can easily check that $(u, \pi, a, \omega)$ satisfies the original system \eqref{eq:mainsys} satisfying \eqref{es:main}.

\bibliographystyle{siam}
%%\bibliography{fish}
\bibliography{INS_FSI-Lp-Lq_MT}

\end{document}